\documentclass[12pt]{article}
\usepackage{verbatim}
\usepackage{amssymb,amsmath}
\usepackage{amsthm}
\usepackage{graphicx}
\usepackage{epsfig}
\newtheorem{corollary}{Corollary}[section]
\newtheorem{lemma}[corollary]{Lemma}
\newtheorem{proposition}[corollary]{Proposition}
\newtheorem{theorem}[corollary]{Theorem}
\newcommand{\Prob} {{\mathbb P}}
\newcommand{\Z}{{\mathbb Z}}
\newcommand{\E}{{\mathbb E}}

\newcommand{\R}{{\mathbb{R}} }

\newcommand{\dist}{{\rm dist}}

\def\H{\mathbb{H}}

\def\R{\mathbb{R}}
\def\Z{\mathbb{Z}}

\def\e{\epsilon}
\def\g{\gamma}
\def\k{\kappa}

\def \Im {{\rm Im}}
\def \Re {{\rm Re}}
\def \p {\partial}

\def \Half {{\mathbb H}}

\def \Disk {{\mathbb D}}

\def \diam {{\rm diam}}

\def \hcap {{\rm hcap}}

\def \F {{\cal F}}

\def \hm {{\rm hm }}

\newenvironment{definition}[1][Definition]{\begin{trivlist}
\item[\hskip \labelsep {\bfseries #1}]}{\end{trivlist}}

\def  \newernot{\Phi}

\begin{document}

\title{Basic properties of the natural parametrization for the
Schramm-Loewner evolution}

\author{Gregory F. Lawler \thanks{Research supported by National
Science Foundation grant DMS-0907143.}\\
University of Chicago  \\ \\ \\
Mohammad A. Rezaei\\ University of Chicago}

\maketitle

\begin{abstract}
The natural paramterization or length for the Schramm-Loewner 
evolution ($SLE_\kappa$) is the candidate for the scaling limit of the  length 
of discrete curves for
$\kappa < 8$.  We improve the proof of the existence of the
parametrization and use this to establish some new
results.   In particular,
we show that the natural  parametrization
is independent of domain and
it  is H\"older continuous with respect to the capacity
parametrization.     We also give up-to-constants bounds
for the two-point Green's function.  Although we do not
prove the conjecture that the natural length is given by
the appropriate Minkowski content, we do prove that
the corresponding expectations converge.

\end{abstract}

\section{Introduction}

 A number of measures on paths or clusters on two-dimensional
lattices arising from critical statistical mechanical
models are believed to exhibit some kind of conformal
invariance in the
scaling limit.
Schramm introduced a one-parameter family of such processes,
now called the
{\em (chordal) Schramm-Loewner evolution with parameter $\kappa$ ($SLE_\kappa$)},
and showed that these give the only possible limits for conformally
invariant processes in simply connected domains satisfying a certain
``domain Markov property''.  He defined the process as a 
probability
measure on curves
from $0$ to $\infty$ in $\Half$, and then used conformal invariance
to define the process in other simply connected domains.

The definition of the process in $\Half$ uses  parametrization
by {\em half-plane capacity} (see Section \ref{SLEdef} for definitions).    Suppose $\gamma:(0,t] \rightarrow \overline{\Half}$ is a
(non-crossing) curve
parameterized so that $\hcap(\gamma(0,t]) = at$ for some constant $a > 0$.
We write $\gamma_t$ for the set of points $\gamma(0,t]$.
Let $H_t$ denote the unbounded component of $\Half \setminus \gamma_t$
and $g_t: H_t \rightarrow \Half$ be
the unique conformal transformation with
$g_t(z) - z = o(1)$ as $z \rightarrow \infty$.  Then the following holds.
\begin{itemize}
\item  For $z \in \Half$, the map $t \mapsto g_t(z)$ is a smooth flow and satisfies
the Loewner differential equation
\[                  \p_t g_t(z) = \frac{a}{g_t(z) - U_t} , \;\;\;\; g_0(z) = z , \]
where $U_t$ is a continuous function on $\R$.
\item If $t,s > 0$, and $\eta(r) = g_t(\gamma(t+r))$,
\[             as = \hcap\left(\gamma_{t+s}\right) - \hcap\left(\gamma_t
\right) = \hcap\left(\eta_s\right). \]
\end{itemize}
Schramm defined chordal $SLE_\kappa$ to be the solution to the Loewner equation
with $a = 2$ and $U_t$ a Brownian motion with variance parameter $\kappa$.
An equivalent definition (up to a linear time change)
 which we use in this paper is to choose $U_t$ to be
a standard Brownian motion and $a = 2/\kappa$.  It has been shown that
a number of discrete random models have $SLE$ as the scaling limit provided
that the discrete models are parameterized using (discrete) half-plane
capacity.  Examples are loop-erased random walk for $\kappa=2$
\cite{LSW}, Ising interfaces for $\kappa=3$ \cite{Smir2},
harmonic explorer for $\kappa=4$ \cite{SS}, percolation interfaces on the triangular lattice for $\kappa=6 $ \cite{Smir1},
and  uniform spanning trees for $\kappa=8$
\cite{LSW}.

If $D$ is a simply connected domain with distinct boundary points $z,w$, then
chordal $SLE_\kappa$ from $z$ to $w$ in $D$ is defined by taking the conformal
image of $SLE_\kappa$ in the upper half plane under a transformation
$F: \Half \rightarrow D$ with $F(0) = z, F(\infty) = w$.  The map $F$ is not
unique, but scale invariance of $SLE$ in $\Half$ shows that the distribution on
paths is independent of the choice.  This can be considered as a measure
on the curves $F \circ \gamma$
with the induced parametrization  or as a measure
on curves modulo reparameterization.


While the capacity parametrization is useful for analyzing the curve, it is not
the scaling limit of the ``natural'' parametrization of the discrete models.  For
example, for loop-erased walks, it is natural to parameterize by the length of
the random walk.  One can ask whether the curves parameterized by a normalized
version of this ``natural length'' converge to $SLE$ with a different parametrization.
The Hausdorff dimension of the $SLE$ paths \cite{Bf} is  $d=1+\rm{min}\{\frac{\kappa}{8},1\}$.
It is conjectured, but still unproven, that the ``natural length'' of an $SLE$ path can be
given by an appropriate $d$-dimensional ``measure''.  If $\k \geq 8$, then the
paths are plane-filling, and we can choose the measure of $\gamma_t$ to be the
area of $\gamma_t$.  For the remainder of this paper we consider the case $\k < 8$
for which $1 < d < 2$.

\def \mcon { {\rm Cont}_d}
A candidate for the natural length of $\gamma(0,t]$ is the $d$-dimensional
Minkowski content defined as follows.  Let $f(\epsilon)$
be a positive function with $f(\epsilon) \rightarrow
0$ and $\epsilon/f(\epsilon)
\rightarrow 0$ as $\epsilon \downarrow 0$.  Let
\[  \mcon(\gamma_t;\epsilon,f) =  \epsilon^{d-2}
                     {\rm Area} \left\{z: \dist(z,\gamma_t) \leq \epsilon,
                     \dist(z,\p D) \geq f(\epsilon) \right\}, \] and
\begin{equation}  \label{oct12.2}
              \mcon(\gamma_t) =   
               \lim_{\epsilon \rightarrow 0+} \mcon(\gamma_t;\epsilon,f)  ,
                     \end{equation}
   provided that the limit exists and is independent of $f$.
It is not known whether or not this limit exists.  For the moment  let
us assume that it does  and,
moreover, that
the function $t \mapsto  \mcon(\gamma_t)  $ is continuous and strictly increasing.
In this case we can reparameterize $\gamma$ by
\[       \tilde \gamma(t) = \gamma(\sigma_t), \;\;\;\;\
  \sigma_t = \inf\{s: \mcon(\gamma_s) = t \}. \]
 We could now {\em define} $SLE_\kappa$ to be the measure on curves
 $\tilde \gamma$ with this ``natural'' parametrization.  Among the
 properties that this would have are the following.
 \begin{itemize}
 \item  Suppose $\tilde \gamma_t$ is an initial segment of an
 $SLE_\kappa$ from $0$ to $\infty$ in $\Half$ and $\tilde \gamma_t
 \subset D \subset \Half$.  If we consider $\tilde \gamma_t$ as being an $SLE_\kappa$
 path in $D$ instead, the amount of time to traverse $\tilde{\gamma}_t$
 is the same in $D$ as in $\Half$.
 \item Suppose $D$ is a simply connected domain with distinct boundary
 points $z,w$; $F: \Half \rightarrow D$ is a conformal transformation
 with $F(0) = z, F(\infty) = w$; and $\tilde \gamma(t)$ is an $SLE_\kappa$ curve
 in $\Half$ with the natural parametrization.   Then
 \[  \mcon(F \circ \tilde \gamma_t) = \int_0^t |F'(\tilde \gamma(s))|^d
  \, ds. \]
  In particular, to define $SLE_\kappa$ in $D$ with the natural parametrization,
  one lets
  \[         \tilde \eta(t) =   F \circ \tilde \gamma(\sigma_t) , \]
  where $\sigma_t$ is defined by
  \[           \int_0^{\sigma_t} |F'(\tilde \gamma(s))|^d
  \, ds = t  \]
  \item  Suppose $D$ is a bounded domain with locally analytic
  boundary points $z,w$,  and  $\gamma$ is an $SLE$ path
  from $z$ to $w$ (using any parametrization).  Let $\gamma^R$
  denote the reversed path from $w$ to $z$.  Then
  \[            \mcon(\gamma) = \mcon(\gamma^R) < \infty. \]
  Moreover,
\begin{equation}  \label{oct12.3}
 \E[\mcon(\gamma)] = c \int_D G_D(\zeta;z,w)\,
  dA(\zeta) ,
  \end{equation}
  where $c$ is a constant (depending only on $\kappa$),
  $dA$ denotes integration with respect to area,
 and $G_D(\zeta;z,w)$ denotes the ``Green's function'' for
 $SLE_\kappa$ in $D$.  The function $G_D$,
  whose definition is recalled in Section \ref{SLEdef},
  satisfies
  \begin{equation}  \label{oct12.1}
  \lim_{\epsilon \downarrow 0} \epsilon^{d-2}
   \, \Prob\{\dist(z,\gamma_\infty) \leq \epsilon\}
     = \hat c \, G(z) , 
  \end{equation}
  where $\hat c = \hat c_\kappa \in (0,\infty)$.

 \end{itemize}

It is still open to prove \eqref{oct12.2}.  It is known
that if we choose
\[       G(z) := G_\Half(z;0,\infty) =  \Im(z)^{d-2} \,
    [\sin \arg(z)]^{4a - 1} , \]
    then an analogue of \eqref{oct12.1} holds where
  distance is replaced with conformal radius.   One
  of the goals of this paper is to establish \eqref{oct12.1}
 although our proof does not determine the constant. 
(See Section \ref{SLEdef} for precise statements of what is
known and what we prove here.)
  For other simply connected domains, the
  Green's function can be computed using
  the scaling rule
  \[         G_D(\zeta;z,w) = |f'(\zeta)|^{2-d}
    \, G_{f(D)}(f(\zeta);f(z),f(w)). \]

In \cite{LS}, a different approach was taken to constructing the
natural parametrization, using \eqref{oct12.3} as the starting point.
For ease suppose that $D$ is a bounded domain and $z,w$ are
distinct boundary points with
\[                \Psi = \int_D G_D(\zeta;z,w)\,  d A(\zeta) < \infty . \]
Let $\gamma$ denote an $SLE_\kappa$ curve from $z$ to $w$
in $D$. Let us give the curve the capacity parametrization
inherited from capacity in $\Half$.  Let $\Theta_t$
denote the ``natural length'' of $\gamma_t$ (we expect that
this is a multiple of $\mcon[\gamma_t]$ but we do not
assume it as such).   Then
\[    \E[\Theta_\infty \mid \gamma_t] = \Theta_t + \E\left[\Theta_\infty
  - \Theta_t \mid \gamma_t\right]. \]
  Using \eqref{oct12.3}, we see that
  we would expect
\[  \E\left[\Theta_\infty
  - \Theta_t \mid \gamma_t\right] = \Psi_t :=  \int_{D_t} G_{D_t}(\zeta;
   \gamma(t),w) \, dA(\zeta), \]
where $D_t$ denotes the component of $D \setminus \gamma_t$
that contains $w$ in its boundary.  For each $\zeta$ one can
see that
\[    M_t(\zeta) := G_{D_t}(\zeta;
   \gamma(t),w)   \]
   is a positive local martingale and hence is a supermartingale.
 Therefore, $\Psi_t$ is a supermartingale.  However,
 $N_t := \E[\Theta_\infty \mid \gamma_t] $ should be a martingale.
Therefore, we {\em define} $\Theta_t$ to be the unique increasing
process such that
\[       \Psi_t + \Theta_t \]
is a martingale.  This is  a standard Doob-Meyer
decomposition.

In order to justify this definition, one needs to prove moment bounds.
Indeed, if $\Psi_t$ were actually a local martingale (which would not
be shocking since it is an integral of local martingales), then there would
be no nontrivial increasing process that we could add to $\Psi_t$
to make it a martingale.
\begin{itemize}
\item  In \cite{LS}, it was shown that for $\kappa < 5.0\cdots$, the process
$\Theta_t$ exists in $\Half$ (the definition has to be modified slightly
in $\Half$ because $\Psi_0$ as we have defined it above is infinite---
this is not very difficult).  The necessary second moment bounds
were obtained using the reverse Loewner flow.  It was shown that
for this range of $\kappa$, there exists $\alpha_0 = \alpha_0(\kappa)
>0$ such that the function $t \mapsto \Theta_t$
is H\"older continuous of order $\alpha$ for $\alpha<\alpha_0$ .
\item In \cite{LZ}, the natural parametrization was shown to exist
for all $\kappa < 8$.  There the necessary two-point estimates
were obtained from estimates on the ``two-point Green's function''
\cite{Bf,LW}.  However, the estimates were not strong enough to
determine H\"older continuity of the function $\Theta_t$.

\end{itemize}

This paper continues the study of the natural parametrization
which in turn leads to further study of the multi-point Green's
function.  We extend the definition of the Green's function and
prove some important bounds; we describe these results
in the next section which outlines the paper.  The main new
results about the natural parametrization are the following.

\begin{itemize}

\item  We improve the proof in \cite{LZ} by establishing
that
for all $\kappa < 8$, the discrete
approximations of the natural parametrization
 converge in $L^1$.  Moreover, we show
 that  there exists $\alpha_0 > 0$ such that the
$t \mapsto \Theta_t$ is H\"older continuous of order
$\alpha < \alpha_0$.

\item  We prove that the natural parametrization is ``independent
of domain''.  In other words, we prove the first property that
we listed that the  ``natural parametrization'' should satisfy.

\end{itemize}

We do not establish the reversibility of the
natural parametrization.  It is our hope that the results in this
paper will help us in establishing the limit \eqref{oct12.2}.

\subsection{Overview of the paper}

We start with a review of the Schramm-Loewner evolution (SLE) and the
notation we will use in Section \ref{SLEdef}.  This includes the
definition of the Green's function $G(z)$ and its relationship to the
probability of $SLE$ getting close to a point.  Roughly speaking,
$G(z)$ is the normalized probability that $SLE$ hits $z$.
 The next subsection
introduces the time-dependent Green's function $G^t(z)$ which
corresponds (again, roughly) to the probability that the $SLE$
path hits $z$ by time $t$ in the capacity parametrization.  This function
appears implicitly in \cite{LS} as $G(z) \, \phi(t;z)$, but we find it
useful to formalize this and to prove some estimates.
Section \ref{two-pointsec} studies the two-point Green's function
as introduced in \cite{LW} and defines a time-dependent version of
it.   Two important estimates are stated in this section.  Theorem
\ref{sept8.theorem}, which we label as a theorem because we believe the
estimate will be useful for others, completes the work in \cite{Bf,LW} by
giving a two-sided up-to-constants estimate for the (time
independent) two-point
Green's function.  Lemma \ref{may14.lemma1} gives a time-dependent
version of a two-point estimate from \cite{LZ}.  Both of these
estimates are important in our study of the natural parametrization.
We delay the proofs of these estimates to  
Section \ref{boundssec}.  We do derive some corollaries of these
estimates in Section \ref{two-pointsec}.

We define the natural parametrization in $\Half$ in Section
\ref{nathalfsec}.  Although the definition is the same as that
in \cite {LS,LZ}, we phrase the definition
in terms of  the time-dependent Green's function.   In the next
subsection we prove  the
existence and H\"older continuity of the parametrization.  Our
proof combines ideas in \cite{LS,LZ} as   well as
the H\"older continuity of the $SLE$ path for $\kappa < 8$.
Section \ref{usefulsec}  proves two lemmas that are used in
the subsequent subsection to establish the independence
of the natural parametrization and the domain.  An exact
statement of the independence is given in Theorem \ref{indyprop}.

Section \eqref{boundssec} gives proofs of two of the main
estimates.  These results generalize results from previous
papers, and the arguments rely on the work in those papers.
Lemma \ref{may14.lemma1} extends a result in \cite{LZ} 
to time-dependent Green's functions and uses one fact from
that paper. 
Proposition   \ref{nov17.prop1}, which is an important step in the proof of Theorem \ref{sept8.theorem}, extends a result in
\cite{LW}.   The main extension is to allow the interior target
points to close to the boundary.  

The final section gives the proof of \eqref{oct12.1}.  This proof 
uses properties of two-sided radial $SLE$ but 
is independent of the other results in this paper.

\section{Green's functions for $SLE$}  

\subsection{Schramm-Loewner evolution (SLE) and notation}  \label{SLEdef}

In this section we  review   the Schramm-Loewner evolution and the
chordal Green's function.
  See \cite{Law1,Law2} for   more details.

Suppose  that $\gamma:(0,\infty) \rightarrow
\Half =\{x+iy: y > 0\}$
is a  curve with $\gamma(0+) \in
\R$ and $\gamma(t) \rightarrow \infty$ as $t \rightarrow
\infty$.  Let $H_t$ be the unbounded component of
$\Half \setminus \gamma(0,t]$.   Using the Riemann mapping
theorem, one    can see that there is a unique conformal
transformation
\[            g_t: H_t \longrightarrow \Half \]
satisfying $g_t(z) - z \rightarrow 0$ as $z \rightarrow \infty$.
It has an expansion at infinity
\[           g_t(z) = z + \frac{a(t)}{z} + O(|z|^{-2}). \]
The coefficient $a(t)$  equals
$\hcap(\gamma(0,t])$ where $\hcap(A)$ denotes
the half-plane capacity from infinity of a bounded set
$A$.  There are a number of ways of defining $\hcap$, e.g.,
\[             \hcap(A) = \lim_{y \rightarrow \infty}
   y\,  \E^{iy}[\Im(B_{\tau})], \]
where $B$ is a complex Brownian motion and $\tau = \inf\{t:
B_t \in \R \cup A\}$.

We assume that $\gamma$ is a non-crossing curve; by this we
mean that
  for each $t$,
$\gamma(t) \in \p H_t$ and the image $V_t = g_t(\gamma(t))$
is well defined and a continuous function of $t$.  If $\gamma$ is simple,
then it is non-crossing, but there are non-simple, non-crossing
curves.
Then $g_t$ satisfies the
  {\em  (chordal) Loewner equation}
\begin{equation}  \label{loew}
       \dot g_t(z) = \frac{a}{g_t(z) - V_t} , \;\;\;\;
   g_0(z) = z,
\end{equation}
where $V_t = g_t(\gamma(t))$ is a continuous function.

Conversely, one can start with a continuous
real-valued function $V_t$  and define $g_t$ by \eqref{loew}.
For $z
\in \Half \setminus \{0\}$, the function $t \mapsto g_t(z)$ is
well
defined up to time $T_z := \sup\{t: \Im[g_t(z)]> 0\}$.
The {\em (chordal) Schramm-Loewner evolution (SLE) (from $0$
to $\infty$ in $\Half$)} is the solution to \eqref{loew}
where  $V_t = - B_t$ is a standard Brownian motion and
$a = 2/\kappa$.
There exists a random non-crossing curve $\gamma$, which
is also called $SLE$,
 such that $g_t$ comes from the curve
$\gamma$ as above.
Moreover, if $H_t$ denotes the unbounded component
of $\Half \setminus \gamma(0,t]$, then
\[   H_t = \{z \in \Half: T_z > t\}. \]
 We write \[    f_t(z) = g_t^{-1}(z + V_t). \]


 Throughout this paper we set $ \k = 2/a$.
There are three phases for the curve \cite{RS}: for $\k \leq 4$,
the curve is simple; if $\kappa \geq 8$, the curve is plane-filling;
and for $4 < \k  < 8$, the curve has self-intersections but is
not space-filling.  We will consider $\kappa < 8$ in this paper
in which case \cite{Bf} the Hausdorff dimension of $\gamma_t$
is
\begin{equation}  \label{nov28.20}
       d = 1 + \frac \kappa 8 = 1 + \frac 1{4a}.
\end{equation}

We will use the scaling property of $SLE$ that we recall in a proposition.

\begin{proposition}  Suppose $U_t$ is a standard
Brownian motion and let $g_t$ be the solution
to the Loewner equation \eqref{loew} with $V_t = U_t$
producing the $SLE_\kappa$ curve $\gamma(t)$.
Let $r > 0$ and define
\[   \hat \gamma(t) = r^{-1} \, \gamma(r^2 t), \;\;\;\;
    \hat g_t(z) = r^{-1} \, g_{r^2t}(rz), \;\;\;\;
       \hat U_t = r^{-1} \, U_{r^2t}. \]
Then $\hat \gamma(t)$ has the distribution of $SLE_\kappa$.
Indeed, $\hat g_t(z)$ is the solution to
\eqref{loew} with   $V_t = \hat U_t$.
\end{proposition}

$SLE_\kappa$ in other simply connected domains is
 defined by conformal
invariance.  To be more precise, suppose that $D$ is a simply connected
domain and $w_1,w_2$ are distinct points in $\p D$.  Let $F: \Half
\rightarrow D$ be a conformal transformation of $\Half$ onto $D$ with
$F(0) = w_1, F(\infty) = w_2$.  Then the distribution of
\[        \tilde \gamma(t) = F \circ \gamma(t) , \]
is that of $SLE_\kappa$ in $D$ from $w_1$ to $w_2$.  Although
the map $F$ is
not unique, scale invariance of $SLE_\kappa$ in $\Half$ shows that the
distribution is independent of the choice. This measure is often considered
as a measure on paths modulo reparameterization, but we can also consider
it as a measure on parameterized curves.

If $\gamma(t)$ is an $SLE_\kappa$ curve with transformations $g_t$
and driving function $U_t$, we write $\gamma_t = \gamma(0,t], \gamma
= \gamma_\infty$, and let $H_t$ be the unbounded component of
$\Half \setminus \gamma_t$.  If $z \in \Half$ and $t < T_z$, we let
\begin{equation}  \label{may24.1}
Z_t(z) = g_t(z) - U_t, \;\;\;\; S_t(z) = \sin \left[\arg Z_t(z)\right], \;\;\;\;
  \Upsilon_t(z) = \frac{\Im[g_t(z)]}{|g_t'(z)|}.
  \end{equation}
  If $Z_t(z) = X_t(z) + i Y_t(z)$, then the Loewner equation can
  be written as
  \[  dX_t = \frac{a\, X_t(z)}{|Z_t(z)|^2} \, dt + dB_t, \;\;\;\;
      \p_t Y_t(z) =  - \frac{a \, Y_t(z)}{|Z_t(z)|^2} , \]
 where $B_t = -U_t$.
More generally, if $D$ is a simply connected domain and $z \in D$, we
let $\Upsilon_D(z)$ denote $(1/2)$  times the conformal radius of
$D$ with respect to $z$, that is, if $F:\Disk \rightarrow D$ is a
conformal transformation with $F(0) = z$, then $|F'(0)| = 2\Upsilon_D(z)$.
Using the Schwarz lemma and the Koebe $(1/4)$-theorem we see that
\[  \frac{\Upsilon_D(z)}{2} \leq \dist(z,\p D) \leq 2 \, \Upsilon_D(z) . \]
It is easy to check that if $t < T_z$, then $\Upsilon_t(z)$ as given
in \eqref{may24.1} is the same as $\Upsilon_{H_t}(z)$.  Also,
if $z \not\in \gamma$, then
  $\Upsilon(z) := \Upsilon_{T_z-}(z) = \Upsilon_D(z)$
where $D$ denotes the connected component of $\Half \setminus
\gamma$ containing $z$.

Similarly, if $w_1,w_2$ are distinct boundary points on a simply
connected domain $D$
and $z \in D$, we define
\[            S_D(z;w_1,w_2) = \sin[\arg f(z)] , \]
where $f: D \rightarrow \Half$ is a conformal transformation with $f(w_1) = 0,
f(w_2) = \infty$.  If $t < T_z$, then
$S_t(z) = S_{H_t}(z;\gamma(t),\infty)$.  If $f:D \rightarrow f(D)$ is
a conformal transformation,
\[     S_D(z;w_1,w_2) = S_{f(D)}(f(z);f(w_1),f(w_2)). \]
If $\p_1,\p_2$ denote the two components of $\p D \setminus
\{w_1,w_2\}$, then
\begin{equation}  \label{hmeasure}
    S_D(z;w_1,w_2)  \asymp \min\left\{\hm_{D}(z,\p_1),
\hm_D(z,\p_2) \right\}.
\end{equation}
Here, and throughout this paper, $\hm$ will denote   harmonic measure; that is,
 $\hm_D(z,K)$ is the probability that a Brownian
motion starting at $z$ exits $D$ at $K$.

  Let
\begin{equation}  \label{green}
     G(z)
   =|z|^{d-2}  \, \sin^{\frac
  \kappa 8 + \frac 8{\kappa} -2} (\arg z)  = \Im(z)^{d-2}
    \, \sin^{4a-1}(\arg z)  ,
\end{equation}
denote the {\em (chordal) Green's function for $SLE_\kappa$ (in
$\Half$ from $0$ to $\infty$)}.  This function first appeared  in \cite{RS}
and the combination $(d,G)$
can be  characterized up to multiplicative constant
by the scaling rule $G(rz) =
r^{d-2} \, G(z)$ and the fact that
\begin{equation}  \label{localmart}
    M_t(z) := |g_t'(z)|^{2-d} \, G(Z_t(z))
\end{equation}
is a local martingale.
More generally, if $D$ is a simply connected domain with
distinct $w_1,w_2 \in \p D$, we define
\[    G_D(z;w_1,w_2) = \Upsilon_D(z)^{d-2} \ S_D(z;w_1,w_2)^{4a-1} . \]
The Green's function satisfies the conformal covariance rule
\[   G_D(z;w_1,w_2) = |f'(z)|^{2-d} \, G_{f(D)}(f(z);f(w_1),f(w_2)). \]
Note that if $t < T_z$, then
\[     M_t(z) = G_{H_t}(z;\gamma(t), \infty). \]
The local martingale $M_t(z)$ is not a martingale because
it ``blows up'' at time $t=T_z$.   If we stop it before that time, it
is actually a martingale.  To be precise, suppose that
\begin{equation} \label{nov18.1}
        \tau = \tau_{\epsilon,z} = \inf\{t: \Upsilon_t(z)
 \leq \epsilon\}.
 \end{equation}
 Then for every $\epsilon > 0$, $M_{t \wedge \tau}(z)$ is a
 martingale.  Moreover, on the event $\tau = \infty$, we have
 $M_\infty(z) = 0$.  Therefore, if $\sigma \leq \tau$ is a stopping
 time,
\begin{equation}
\label{oct18.1}
     G(z) = \E\left[M_\sigma(z) \right] = \E\left[|g_\sigma'(z)|^{2-d}
  \, G(Z_\sigma(z))\right].
  \end{equation}

The following is proved in \cite{Law2} (the proof there is in the upper half
plane, but it immediately extends by conformal invariance).

\begin{proposition} \label{mar14.prop1} Suppose $\kappa < 8$,
$z \in D, w_1,w_2 \in \p D$ and $\gamma$
is a chordal $SLE_\kappa$ path from $w_1$ to $w_2$ in $D$.  Let
$D_\infty$ denote the component of $D \setminus \gamma$ containing
$z$.  Then, as $\epsilon \downarrow 0$,
\[   \Prob\{\Upsilon_{D_\infty}(z) \leq
           \epsilon  \} \sim
  c_* \, \epsilon^{2-d} \, G_D(z;w_1,w_2) , \;\;\;\;
   c_*  = 2\,\left[\int_0^\pi  \sin^{4a}x \, dx\right]^{-1} . \]
\end{proposition}

In this paper, we give the analagous result replacing
conformal radius with distance to the curve.  We do not
have an explicit form for the constant $\hat c$.   

\begin{theorem} \label{jul23.theorem1}
 Suppose $\kappa < 8$.  There exist
 $0 < \hat c,c,u  < \infty $ (depending
 on $\kappa$) such that the following holds.
 Suppose $D$ is a simply connected domain,
$z \in D$, $ w_1,w_2\in 
 \p D$ and $\gamma$
is a chordal $SLE_\kappa$ path from $w_1$ to $w_2$ in $D$.  Then, as $\epsilon \downarrow 0$,
\begin{equation}  \label{apr2.4}   \Prob\{\dist(z,\gamma_\infty) \leq
           \epsilon  \} \sim
  \hat c \, G_D(z;w_1,w_2) \, \epsilon^{2-d},
  \end{equation}
  Moreover, if  
   $\epsilon <\dist(z,\p D)/10$,
  \[ \left| \epsilon^{d-2} \,
   G_D(z;w_1,w_2)^{-1} \, \Prob\{\dist(z,\gamma_\infty)
    \leq \epsilon \}
    - \hat c\right| \leq c \, \left( \frac{\epsilon}{
    \dist(z,\p D)} \right)^{u}. \]
\end{theorem}

\begin{proof}  See Section \ref{newsection}.\end{proof}

\begin{corollary}
Suppose $\kappa < 8$.  Let $f(r)$ be a positive function with
$f(r)  \rightarrow  0$ and $r/f(r) \rightarrow
0$ as $r \downarrow 0$.
Suppose $D$ is a simply connected domain
and 
 $ w_1,w_2 \in \partial D$. Let $\gamma$ be
 an $SLE_\kappa$ path from $w_1$ to $w_2$ in
 $D$ and let
 \[  Y_\epsilon = {\rm Area}\{z \in D:
  \dist(z,\p D) \geq f(\epsilon), \dist(z,\gamma_\infty)
  \leq \epsilon\} . \]
  Then
  \[  \lim_{\epsilon \downarrow 0}
      \epsilon^{d-2} \, \E\left[Y_\epsilon
      \right] =  \hat c \int_DG(z;w_1,w_2) \, dA(z), \]
      where $\hat c$ is the constant in \eqref{apr2.4}.
  
 \end{corollary}

\subsection{Time dependent Green's function}  \label{timesec}

If $z \in \Half$, then {\em two-sided radial $SLE_\kappa$ (stopped at $T_z$,
the time at  which it
reaches $z$)} is chordal $SLE_\kappa$ ``conditioned to go through $z$''.  This is
conditioning on an event of measure zero, but there are several equivalent ways
to make this precise.  The term two-sided radial comes from the fact that
at the interior point $z$, there are two paths coming out, $\gamma[0,T_z]$ and
$\gamma[T_z,\infty)$.   We are only considering the marginal distribution
on $\gamma[0,T_z]$.

The elegant way to define two-sided radial $SLE_\kappa$ is to say that it
is chordal $SLE_\kappa$ weighted by the local martingale $M_t(z)$
(see, e.g., \cite[Section 2.4]{LW}).  By
the Girsanov theorem, we can also describe this process by giving the
drift on the driving function.   Indeed, if $
\Prob_z^*$ denotes the probability measure of
two-sided radial and $X_t = X_t(z) = \Re[Z_t(z)]$, then
\begin{equation}  \label{oct19.1}
       dX_t = \frac{(1-3a)\, X_t}{|Z_t|^2} \, dt + dB_t, \;\;\;\; X_0 = \Re(z),
   \end{equation}
 where $B_t$ is a $\Prob_z^*$-Brownian motion.  (The equation
 for $Y_t = Y_t(z)$ is deterministic, and hence the same for
 as for usual chordal $SLE$.)  The following technical fact is useful.

\begin{proposition} \cite{Law3} If $\kappa < 8$ and $z \in \Half$, 
then with probability one, two-sided
radial $SLE_\kappa$ is continuous at $T_z$.  In other words, with
probability one, $T_z < \infty$
and
\[            \lim_{t \uparrow T_z} \gamma(t) = \gamma(T_z) = z. \]
\end{proposition}

Fix $z \in \Half$ and let $\tau_\epsilon =\tau_{\epsilon,z}$ be as in \eqref{nov18.1}.
 For any $\delta < \epsilon$, we can define a probability
measure $\nu(\delta,\epsilon)$ on curves $\gamma(0,\tau_\epsilon]$ by considering
$SLE_\kappa$, conditioning on the event $\{\tau_\delta < \infty\}$, and then
viewing this as a measure on the paths $\gamma(0,\tau_\epsilon]$.
The following proposition which was  proved in \cite{LW}
gives justification to calling two-sided radial $SLE$, ``chordal
$SLE$ conditioned to go through $z$''.

\begin{proposition}
   As $\delta \rightarrow 0+$, the measures
$\nu(\delta,\epsilon)$ converge to two-sided radial $SLE_\kappa$
stopped at
  time $\tau_\epsilon$.
\end{proposition}

We write $\Prob_z^*,\E_z^*$ for probabilities and expectations
with respect to two-sided radial $SLE$ going through $z$.
Let $\phi(z;t)= \Prob_z^*\{T_z \leq t\}$
 denote the distribution function of $T_z$ under the measure
of two-sided radial $SLE$.  The scaling property of $SLE$
implies that $\phi(rz;r^2t) = \phi(z;t)$.
The {\em time-dependent Green's function}
$G^t$ is defined by
\[               G^t(z) = G(z) \, \phi(z;t) . \]
An equivalent way of defining $G^t$ is by
\begin{equation}  \label{apr23.5}
             \E[M_t(z)] = G(z) - G^t(z) .
             \end{equation}
We have the scaling rule
\[    G^{r^2t}(rz) = G(rz) \, \phi(rz;r^2 t) = r^{d-2} \, G(z) \, \phi(z,t)
   = r^{d-2} \, G^t(z) . \]
   The analogue of \eqref{oct18.1} is the following: if $\epsilon > 0$
   and $\sigma \leq \tau_{\epsilon,z}$ is a stopping time, then
\begin{equation}  \label{oct18.2}
 G^t(z) =\E\left[|g_\sigma'(z)|^{2-d}  \, G^{t - \sigma}(Z_\sigma(z))\right]
  = \E\left[|g_\sigma'(z)|^{2-d}  \, G^{t - \sigma}(Z_\sigma(z));
   \sigma < t\right],
 \end{equation}
 where we set $G^s(z) = 0$ if $s \leq 0$.

 While the Green's function $G$ is independent of the
 parametrization of the $SLE$ path, the time-dependent
 function $G^t$ is defined in terms of the half-plane capacity.
 We finish this section by proving some lemmas about $G^t$.
 It follows from the Loewner equation (see, e.g., \cite[Lemma 4.13]{Law1}),
 that in time $t$, the diameter of the path $\gamma_t$ is bounded
 above by
 \[              2a \, \left[t^{1/2}   \vee
 \max\{|U_s| : 0 \leq s \leq t \}\right]. \]
 If $U_t$ is a Brownian motion, it is standard to
 use the reflection principle to show that
\begin{equation}  \label{reflection}
  \Prob\{  \max\{|U_s| : 0 \leq s \leq t \}\geq rt^{1/2}\}
  \leq c \, e^{-r^2/2}.
 \end{equation}
 Hence we also get tail estimates for $\diam[\gamma_t]$.  However,
 for $G^t$, we need tail estimates ``conditioned that $T_z < \infty$'',
 which we  prove in the next lemma. 
 The exponents in the estimates
 are not optimal but they will suffice for our needs.


\begin{lemma} \label{may19lemma}
 There exists $c < \infty$ such that if
$z=x+iy$, then
\[   G^t(z) = 0 , \;\;\;\;   y^2 \geq 2at , \]
\[            G^t(z) \leq c \, e^{-\frac{x^2}{8t}} \, G(z) .\]
Moreover, there exists $c_1 > 0$ such that
 if $  y^2 \leq 2at$,
\[            G^{2t}(z) \geq c_1 \, e^{-\frac{5x^2}{t}} \, G(z) , \]
\[      G^{100t} (z) \geq c_1 \, e^{\frac{x^2}{40t}} \, G^t(z) . \]
\end{lemma}

\begin{proof}  Without loss of generality we assume that $x \geq 0$.
The first inequality follows immediately from the Loewner
equation since $\p_t \left([\Im g_t(z)]^2 \right)\geq -2a.$
For the remainder of the proof we assume $y^2 \leq 2at$.

Let $\Prob^* = \Prob^*_z$ and let $X_t, Z_t, B_t$ be
as in  \eqref{oct19.1} with $X_0 = B_0 =  x$. Let $T_z = \inf\{t:
Z_t = 0\} = \inf\{t: \gamma(t) = z\}$.
Then the second
and third inequalities are equivalent to
\begin{equation}  \label{nov11.1}
    \Prob^*\{T_z \leq t\} \leq c \, e^{-\frac{x^2}{8t} }.
        \end{equation}
  \begin{equation} \label{nov11.2}
  \Prob^*\{T_z \leq 2t\} \geq c_1 \, e^{-\frac{5x^2}{t}}.
 \end{equation}

By scaling, it suffices to prove \eqref{nov11.1} with $x=1$
and for $t$ sufficiently small,
Let $R = \inf\{t: X_t = 1/4\}, S= \inf\{t:B_t = 1/2\}$.  Note that if $t \leq R$, then
$ B_t \leq X_t + 4rt$ where $r= (3a-1) \vee 0$. Therefore,
if $t \leq 1/(16r)$,
\[
\Prob^*\{T_z \leq t\}   \leq   \Prob^*\{R \leq t \} \leq
\Prob^*\{S \leq t\}. \]
The reflection principle for Brownian motion implies that
\[ \Prob^*\{S \leq t\} \leq \Prob^*\left\{\min_{0 \leq s \leq t}( B_s-B_0) \leq
- 1/2\right\} \leq 2 \Prob\{B_t \geq 1/2\}   \leq c \, e^{-\frac{1}{8t}}. \]

By scaling it suffices to prove
 \eqref{nov11.2}  for $t=1$.
 The estimate is easy to establish if $x=0$; indeed,
there exists $\delta  > 0$ such that if $z=iy$ with $y \leq \sqrt{2a}$,
\[ \Prob^*\{T_z \leq 3/2\} \geq  \delta. \]
If $x > 0$, let
\[            T = \inf\{t: X_t = 0 \}. \]
The strong Markov property implies that $\Prob^*\{T_z \leq 2\}
\geq \delta \, \Prob^*\{T \leq 1/2\}$.  Hence,
  it suffices to prove that
$      \Prob^*\{T \leq 1/2 \mid X_0 = x\} \geq c_1 \, e^{- 5x^2}, $
or equivalently, by scaling,
\[     \Prob^*\{T \leq t \mid X_0 = 1\} \leq c_1 \, e^{-\frac{5}{2t}}. \]
Using \eqref{oct19.1} and $a > 1/4$, we can see that it suffices to establish the
last inequality where $X$ satisfies the Bessel equation
\[         dX_t = \frac{dt}{4X_t} + d B_t, \;\;\;\; X_0 = 1 , \]
and $T =\inf\{t: X_t=0\}$.
 For the remainder of this argument
we assume this with $B_0 = 1$
and write just $\Prob$ for the
probability measure.    Let $S = \inf\{t: X_t = 1/2\}$.  Then the
strong Markov property and Bessel scaling imply that $T$ has
the same distribution as $S + (\tilde T/4)$ where $S,\tilde T$
are independent and $\tilde T$ has the same distribution as $T$.
Therefore,
\[   \Prob\{T \leq t \} \geq \Prob\{S \leq t/2\}
\, \Prob\{\tilde T \leq t/2\} = \Prob\{S \leq t/2\} \,
  \Prob\{T \leq 2t\}. \]
  If $s \leq S$, then $X_s \leq B_s + \frac s2$.  Therefore, if
  $t \leq 2$,
  \[ \Prob\{S \leq t/2\} \geq \Prob\{B_{t/2} \leq 0\}
    \geq c \, \sqrt t \, e^{-\frac{1}{t} }  .\]
  In particular, there exists $t_0 >0$, such that for $t \leq t_0$,
\begin{equation} \label{nov11.3}
     \Prob\{T \leq t \} \geq e^{-\frac{5}{4t}} \, \Prob\{T
  \leq 2t\}.
  \end{equation}
Since $\Prob\{T \leq t_0/2\} > 0$, there
 exists $c_1$ such that for $t_0/2 \leq t \leq t_0$,
\begin{equation}  \label{nov11.4}
  \Prob\{T \leq t\} \geq c_1 \, e^{-\frac{5}{2t} }.
  \end{equation}
Using \eqref{nov11.3} and induction, we see that \eqref{nov11.4}
holds for all $t \leq t_0$.   This finishes the proof of
the third inequality.

 The last inequality in the lemma
  follows from the second and third inequalities as follows.
\[
\frac{G^{100 t}(z)}{G(z)}  \geq   c_1 \, \exp \left\{-\frac{5x^2}{50t} \right\}
 = c_1 \, \exp\left\{\frac{x^2}{40 t} \right\} \, \exp\left\{- \frac{x^2}{8t} \right\}
 \geq \frac{c_1}{c} \, \exp\left\{\frac{x^2}{40 t} \right\} \, \frac{G^t(z)}{G(z)}.  \]

\end{proof}

 \begin{lemma}  \label{aug31.lemma1}
   There exists $\beta > 0, c < \infty$ such that if
 $z  = 1 + \delta i$ with $0 < \delta < 1$,  $x \geq 4$, $0 \leq t \leq 1$, and
 \[       \sigma'  = \sigma_{x}' = \inf\{s: |\Re[\gamma(s)]|  =  x \}, \]
 then
 \[   \Prob^*_z\{\sigma' \leq t \wedge T_z \} \leq c \, e^{-\beta x^2/t}. \]
 \end{lemma}

 \begin{proof}  All constants in this proof are independent of
 $x$ and $\delta$ (but may depend on $a = 2/\kappa$).
 Let $\sigma =  \inf\{s: \Re[\gamma(s)] =  x \}$.  We will prove that
 \[   \Prob^*_z\{\sigma\leq t \wedge T_z \} \leq c \, e^{-\beta x^2/t}. \]
  A similar argument, which we omit,
  establishes the estimate for
  $\sigma_- =  \inf\{s: \Re[\gamma(s)] =  - x \}$.

   Let
 \[  Z_s= X_s+iY_s= g_s(z) - U_s = g_s(z) - g(\gamma(s)),\]
 and $Z = X+iY = Z_\sigma$.   As before, let $H_s$ be the
 unbounded component of $\Half \setminus \gamma_s$ and let
 $H = H_\sigma$.
 The Loewner equation \eqref{loew}
  implies that $Y \leq \delta$.
 Suppose that $\sigma < T_z$.  Let $\ell$ denote the vertical line
 segment from $\gamma(\sigma)$ to $\R$,
    ${\mathcal U} = H \setminus \ell,$   and let  ${\mathcal U}'$ be the unbounded
    component of ${\mathcal U}$.  We write $\{\sigma < t \wedge T_z\} = E_1 \cup E_2$ where
 $E_1$ ($E_2$) is the intersection of $\{\sigma <t \wedge T_z\}$ with the event
 that $z \in {\mathcal U}'$  (resp., $z \not\in {\mathcal U}'$) .

On the event $E_1$, deterministic conformal mapping estimates  imply
 that there exist $c_1$ such that
 \[                   X \leq - c_1 x . \]
 (We leave out the details of this estimate which can be obtained by considering
 the event that 
 the path of a  Brownian motion starting at $(x/2) + 10i$ and
 stopped when
 it leaves $H$, 
 concatenated with the half-infinite line $(x/2)+i[10,\infty)$, separates the disk of radius $\dist(z,\partial H)/2$ about $z$  from
 $U_\sigma$. By considering the image of this event under $g_\sigma$
 we can get the estimate.)
 Therefore,
 \[  \Prob_z^*(E_1) \leq \Prob_z^*\{X_s \leq - c_1 x \mbox{ for some }
   s \leq t  \wedge T_z\}. \]
   We know from \eqref{oct19.1} that $X_t$ satisfies
 \[         dX_t = \frac{(1-3a)\, X_t}{|Z_t|^2} \, dt + dB_t, \;\;\;\; X_0 = 1,\]
 where $B_t$ is a $\Prob^*_z$-Brownian motion.   By comparison with
 a Bessel process (we omit the details), we can see that for $y \geq 2$,
 \[ \Prob_z^*\{X_s \leq - y \mbox{ for some }
   s \leq t  \wedge T_z \mid X_0 \geq  1\} \leq c \, e^{-\beta' y^2/t},  \]
   for appropriate $c,\beta'$.

  We now assume $E_2$ holds.
 Let ${\mathcal U}''$ denote the component of ${\mathcal U}$ containing
  $z$.   Except on an event of $\Prob_z^*$ measure zero, we know that
  $\ell \subset \p {\mathcal U}''$.  We claim there exists $c_3,\beta_3$ such that
  for $z' \in {\mathcal U}''$ with $\Re(z') \leq 3/2$, we have
\begin{equation}  \label{aug31.1}
   \hm_{{\mathcal U}''} (z', \ell) \leq c_3 \, e^{-\beta_3 x^2/t}.
   \end{equation}
  To see this, note that we can find a crosscut
  $\ell''$  of the form $(2 + y_1 i, 2 +y_2 i)$ with the
  property that $\ell''$ disconnects $z'$ from $\ell$ in ${\mathcal U}''$.   Since
  $\hcap[\gamma_\sigma] \leq t$, we know (see \cite[Theorem 1]{LLN})
  that the area of ${\mathcal U}''$ is bounded
  by $c_4 t$.  Also, $\dist(\ell,\ell'') = x -2 \geq x/2$.
  Using extremal distance estimates (see, e.g., \cite[Lemma 3.74]{Law1}),
   we know that the extremal distance 
  between $\ell$ and $\ell''$ in ${\mathcal U}'' \setminus \ell''$ is bounded
  below by $c_5 x^2/t$.  The estimate \eqref{aug31.1} follows from the
  relationship between harmonic measure and extremal distance.

 Let $\Upsilon_t = Y_t/|g_t'(z)|$ denote (one half) times the conformal radius
of $z$ in $H_t$.  We know that $\Upsilon_0 = \delta$ and
 $\Upsilon_t \asymp_2
\dist(z,\p H_t)$.  Let $\eta_k = \inf\{t: \Upsilon_t \leq 2^{-k}
 \, \delta \}$, and write
 \[  \Prob_z^*(E_2) = \sum_{k=1}^\infty
   \Prob_z^*(E_2 \cap \{\eta_{k-1}  \leq \sigma  < \eta_{k} \}). \]  
 Using the Beurling estimate, \eqref{hmeasure}, and  \eqref{aug31.1},
  we can see that on the event $E_2 \cap \{\eta_{k-1} \leq \sigma <\eta_{k} \}$,
 \begin{equation} \label{aug31.2}
          S_\sigma \leq  c\, 2^{-k/2} \,  \delta \, e^{-\beta_3 x^2/t} .
          \end{equation}
 By the study of two-sided radial $SLE_\kappa$ in the radial parametrization (see \cite[Section 2.1.1]{LZ})
 we know that there exist $c_6,\alpha$ such that
 \[   \Prob^*_z\{ \exists t \leq \eta_1: S_t \leq \delta \, r\} \leq c \, r^\alpha , \]
 and for $k \geq 1$,
 \[   \Prob^*_z\{\exists \eta _k \leq t \leq \eta_{k+1}: S_t \leq r\} \leq c \, r^\alpha . \]
 Combining this with \eqref{aug31.2}, we get
 \[  \Prob^*_z\left[E_2 \cap\{\eta_k < \sigma \leq \eta_{k-1} \}\right] \leq c \, 2^{-\alpha k/2}
  \, e^{-\alpha \beta_3 x^2/t} . \]
  By summing over $k$, we see that $\Prob^*_z(E_2) \leq c \,
 e^{-\alpha \beta_3 x^2/t} . $
  \end{proof}

  \begin{corollary}  \label{aug31.corollary}
  For every $\epsilon > 0$ there exists $K < \infty$  such that
  the following holds.  Suppose $z = x+iy|x|$ with $ x \neq 0, y \leq 1$ and $x^2y^2/a \leq t \leq x^2$.
  Then
  \[  \Prob_z^*\{\diam[\gamma] \geq Kx  \mid T_z \leq t \} \leq   \epsilon. \]
  \end{corollary}

  \begin{proof}  By scaling and symmetry we may assume $x=1$.  By Lemma
 \ref{may19lemma},  there exist $c_1,\beta_1$ such that
  \[ \Prob_z^*\{T_z \leq t \} \geq c_1 \, e^{-\beta_1/t}.\]  By Lemma \ref{aug31.lemma1},
  there exist $c_2,\beta_2$ such that
  \[    \Prob_z^*\{ T_z \leq t \, , \, \diam[\gamma]  \geq K \} \leq c_2 \, e^{-\beta_2 K^2/t}. \]
  Therefore, if $K \geq  2 \beta_1/\beta_2$,
  \[  \Prob_z^*  \{\diam[\gamma] \geq K \mid T_z \leq t \}  \leq (c_2/c_1)
   \, \exp \left\{ \frac{\beta_1 - \beta_2 K^2 }{t} \right\}  \leq
     (c_2/c_1) \, \exp \left\{- \frac{\beta_2K}{2t} \right\} \]

  \end{proof}

%
%
%
%
%
%

Let
\[    M_{s}^t(z)
             =  |g_s'(z)|^{2-d} \,  G^{t-s}(Z_s(z)) . \]
For fixed $t$, $M_{s}^t(z)$ is a continuous,
nonnegative local martingale
with $M_{s}^t(z) = 0$ for $s \geq  t$.  Moreover,
if $r \leq t$, then
\begin{equation}  \label{Diffmartingale}
 N_s = M_{s}^t(z) - M_{s}^r(z)
  = |g_s'(z)|^{2-d} \, \left[G^{t-s}(Z_s(z)) -
    G^{r-s}(Z_s(z))\right] ,
\end{equation}
is a {\em martingale} for $0 \leq s \leq r$.
In analogy to \eqref{apr23.5} we can see that if  $s \leq t $,
\[   \E\left[M_{s}^t(z) \right] = G^t(z) - G^s(z). \]

Let
\[          L(0,t) = \int G^t(z) \, dA(z) , \]
and more generally, if $s < t$,
\begin{equation}  \label{oct2.5}
             L(s,t) = \int M_s^t(z) \, dA(z)
  = \int |g_s'(z)|^{2-d} \, G^{t-s}(Z_s(z)) \, dA(z).
  \end{equation}
  Here, and throughout this paper, $dA$ refers to
  integration with respect to 2-dimensional Lebesgue
  measure (area) and, unless specified otherwise, integrals
  are over $\Half$.
  Note that
\begin{equation}  \label{apr23.6}
  \E[L(s,t)] = \int [G^t(z) - G^s(z)] \, dA(z) ,
  \end{equation}
and  if $s < t  \leq t_1$,
\begin{equation}  \label{apr23.6.alt}
  L(s,t) = \E\left[L(s,t_1) - L(t,t_1) \mid \F_s\right].
  \end{equation}
For fixed $t$, $M_s^t(z)$ is a positive supermartingale
in $s$.  Therefore, we have the following important property.
\begin{itemize}
\item  For fixed $t$,  $L(s,t)$ is a nonnegative
supermartingale in $s$
with $L(s,t) = 0$ for $s \geq t$.
\end{itemize}

If $f_s(w) = g_s^{-1}(w+U_s)$, then the substitution
$z = f_s(w)$ in \eqref{oct2.5} yields
\[    L(s,s+t) = \int |f_s'(w)|^d \, G^t(w) \, dA(w).\]
If $r >0$, then the scaling rules for $SLE$ imply
\begin{eqnarray}
L(s,s+r^2t) & = &  \int  |f_s'(w)|^d \, G^{r^2 t}(w) \, dA(w)\nonumber
\\
 & = &  r^{d-2}  \int |f_s'(w)|^d \, G^{t}(w/r) \, dA(w)
 \nonumber \\
 & = & r^d \int |f_t'(rz)|^{d} \, G^t(z) \, dA(z) .
 \label{oct2.6}
 \end{eqnarray}

  \subsection{Two-point Green's function}  \label{two-pointsec}

 The proof of existence of $\Theta_t$ in \cite{LZ}
 uses the two-point Green's function $G(z,w)$ introduced
 in \cite{LW}.  Our proof of
 Theorem \ref{maytheorem2} will  use  a time-dependent
 version.

The two-point Green's function $G(z,w)$ is the
normalized probability that $z$ and $w$ are both in
$\gamma(0,\infty)$.  More precisely, it satisfies
\begin{equation}  \label{mar29.2}
     \Prob\left\{\Upsilon_\infty(z) < \epsilon,
\Upsilon_\infty(w) < \delta \right\}
  \sim c_*^2 \, G(z,w) \, \epsilon^{2-d} \, \delta^{2-d}, \;\;\;\;
    \epsilon, \delta \downarrow 0.
    \end{equation}
 It can be written as
\begin{equation}  \label{mar29.3}
  G(z,w) = \hat G(z,w) + \hat G(w,z) ,
  \end{equation}
 where $\hat G(z,w)$, roughly speaking, is the probability to visit
 $z$ first and then $w$ later.  To be precise, we can write
 \begin{equation}  \label{mar29.4}
     \hat G(z,w) = G(z) \, \E_z^* \left[M_{T_z}(w)\right].
     \end{equation}
In \cite{LW} it is shown that if one defines $G(z,w)$ by
\eqref{mar29.3} and \eqref{mar29.4}, then \eqref{mar29.2} holds.
Given \eqref{mar29.2}, we see that there is a corresponding
two-point local martingale
\[   M_t(z,w) = |g_t'(z)|^{2-d} \, |g_t'(w)|^{2-d}
 \, G(Z_t(z),Z_t(w)). \]
Note that $G$ is symmetric in $z,w$  but the ordered Green's
function $\hat G$ is not.  We will
need a bounded time version of the Green's function, and it is easier
to define first the  nonsymmetric version.

\begin{definition}  The Green's function
$\hat G^{s,t}(z,w)$ is defined by
\[      \hat G^{s,t}(z,w) =
    G(z) \, \E_z^* \left[M_{T_z}^{t-T_z}(w) \; ;\;  T_z \leq
     s\right].\]
The function $G^{s,t}(z,w)$ is defined by
\[    G^{s,t}(z,w) = \hat G^{s,t}(z,w)
  + \hat G^{t,s}(w,z).\]
\end{definition}

We will be using the symmetric function
$G^{t,t}(z,w)$.
%
 The rough interpretation of $G^{t,t}(z,w)$ is the normalized
 probability that the path goes through both $z$ and $w$ and
 both of the visits
 occur before time $t$.
   As for the single point Green's
 function, we can give some equivalent formulations that
 follow immediately from the definition and the scaling rules.

 \begin{proposition}  $\;$

 \begin{itemize}

\item If $s < t$,
\[ \hat  G^{t,t'}(z,w) - \hat G^{s,t'}(z,w)
  = \E\left[|g_s'(z) \, g_s'(w)|^{2-d}
    \, \hat G^{t-s,t' - s}(Z_s(z),Z_s(w))\right].\]

   \item  If $r >0$,
   \[   \hat G^{s,t}(z,w)  = r^{2(2-d)} \, \hat G^{r^2s,r^2t}(rz,rw), \]
\[  G^{s,t}(z,w)  = r^{2(2-d)} \,  G^{r^2s,r^2t}(rz,rw).\]

    \end{itemize}

    \end{proposition}

The next two results are  estimates on the
 two-point Green's function that we will use.
 We delay the proofs until later.

 \begin{theorem}  \label{sept8.theorem}
 There exist $0 < c_1 < c_2 < \infty$ such that
  If  $z,w \in \Half$
  with $|z| \leq |w|$,
  then
 \[  c_1 \, q^{d-2}
        \, [S(w) \vee q]^{-\beta} \leq
         \frac{G(z,w)}{G(z)\,  G(w)}   \leq c_2 \,   q^{d-2}
        \, [S(w) \vee q]^{-\beta} , \]
        where
        \[  q = \frac{|w-z|}{|w|} \leq 2 , \;\;\;\; \beta=  (4a-1) - (2-d) = \frac{\kappa} 8 +
        \frac 8 \kappa - 2 > 0 . \]
   \end{theorem}

   \begin{proof} See Section \ref{uglysec2}.
   \end{proof}

 \begin{lemma}  \label{may14.lemma1}

  There exists $c < \infty$ such that
 if   $z,w \in \Half$ and  $s,t > 0$,
 \begin{equation}
 \label{twopointer}
    G^s(z) \, G^t(w) \leq c \, G^{cs,ct}(z,w) .
    \end{equation}
  \end{lemma}

  \begin{proof}
See  Section \ref{uglysec}. \end{proof}

  We will now prove some  consequences
  of these estimates.

   \begin{lemma}
   There exists $c < \infty $ such that if $z,w \in \Half$ with
   $|z| \leq |w|$,
   \begin{equation}  \label{may19.11}
     G(z,w) \leq c \,
 |z| ^{d-2} \, |z-w|^{d-2 }.
 \end{equation}
   \end{lemma}

   \begin{proof}
   \begin{eqnarray*}
   G(z,w) & \leq & c_2q^{d-2} \, [S(w) \vee q]^{-\beta} \, G(z) \, G(w) \\
      & \leq & c \, q^{d-2} \,  [S(w) \vee q]^{-\beta} \, |z|^{d-2}
        \, |w|^{d-2} \, S(w)^{\beta}  \\
          & \leq & c \, |w-z|^{d-2} \, |z|^{d-2}.
        \end{eqnarray*}
        \end{proof}

 \begin{lemma}  \label{nov16.1}
  There exist $c< \infty, \beta > 0$
 such that if $z = x_z + i y_z,w = x_w + i y_w\in \Half$, then
\begin{equation}  \label{may19.12}
  G^{t,t}(z,w) \leq c\, e^{-\beta |w|^2/t}
   \, G(z,w) .
   \end{equation}
\end{lemma}

\begin{proof} Without loss of generality we assume that $|w|^2
 \geq |z|^2$.   By scaling and symmetry, we may assume $|w| = 2$,
 and it suffices to prove the estimate for $t$ sufficiently small.

 We first
suppose that  $|z| \geq
1/2$, and let $\sigma = \inf\{t: |\gamma(t)| = 1/4\}$.  Then
\[  G^{t,t}(z,w) \leq \Prob\{\sigma \leq t\} \,
  \E[M_{\sigma}(z,w) \mid \sigma \leq t]. \]
Using distortion estimates and Theorem \ref{sept8.theorem},
we see that  $M_\sigma(z,w)  \asymp
G(z,w)$, and hence
 \[   G^{t,t}(z,w) \leq c \, \Prob\{\sigma \leq t\} G(z,w). \]
Using the Loewner equation, there exists $\delta >0, t_0 > 0$ such
that if $t \leq t_0$ and $|U_s| \leq \delta, 0 \leq s \leq t$,
then $\sigma
> t$.  Therefore,
\begin{equation}  \label{may22.8}
 \Prob\{\sigma \leq t \} \leq
  \Prob\left\{\max_{0 \leq s \leq t } |U_s| \geq \delta\right\}
         \leq c \, e^{-\beta/t },
         \end{equation}
         for some $c,\beta$.
Similarly, for every $\epsilon > 0$, we can find $c_\epsilon,
\beta_\epsilon$ such that
\begin{equation}  \label{nov15}
      G^{t,t}(z,w) \leq c_\epsilon \, e^{-\beta_\epsilon /t} \, G(z,w), \;\;\;\;
   \epsilon \leq |z| , |w| = 2 .
   \end{equation}

We now suppose $|z| \leq 1/2$ for which Theorem \ref{sept8.theorem}
implies that $G(z,w) \asymp G(z) \, G(w)$.  Let
\[  \xi = \inf\{t: |\gamma(t) | = 1\} , \]
  and note that
\[ G^{t,t}(z,w) \leq  \E[M_\xi^{t , t }(z,w); \xi \leq t] +
      G(z) \, \E^*_z[M_{T_z}^{t}(w) ; T_z < \xi] .\]
      If $T_z < \xi$, then straightforward estimates show that
 $|g_{T_z}'(w)| \asymp 1,  G_{T_z}(w) \asymp G(w)$,
  and hence we can use Lemma \ref{may19lemma} to conclude
\[    \E^*_z[M_{T_z}^t(w)  \mid  T_z < \xi]  \leq c \, e^{-\beta/t} \, G(w).\]
    We will now show that
   \[    \E[M_\xi^{t , t }(z,w); \xi \leq t]\leq  c \, e^{-\beta/t} \, G(z) \, G(w). \]
   By \eqref{nov15},
   it suffices to show this for $t,|z|$ sufficiently small.
    Let $H = H_ \xi  $, $g = g_ \xi $,  $\hat z = g(z) - U_\xi,
    \hat w = g(w) - U_\xi$.
 $r = \dist(z,\p H)/y$.
   We need to show
  \[ \E[|g'(z)|^{2-d} \, |g'(w)|^{2-d}
  \, G^{t,t}(\hat z,\hat w) ; \xi \leq t] \leq c \, e^{-\beta/t} \,G(z) \, G(w) . \]

 On the event $\{\xi \leq t\}$, we know that $\hcap[\gamma_\xi]
 \leq t$.  In particular, for $t$ sufficiently small, $\Im[\gamma(\xi)]
  \leq 1/10$.  Let $\eta_1,\eta_2$ denote the two open subarcs
  of the unit half circle with $\gamma(\xi)$ removed.  Let
  $\eta_2$ be the smaller arc.  Let $\eta$ be the subarc
  that disconnects $z$ from $\infty$ in $H$.  We consider the
  two cases $\eta = \eta_1$ and $\eta = \eta_2$.

   Suppose $\eta = \eta_1$ so that its length is at least
   $\pi/2$.
 Since $|z| \leq 1/2$, we can use
conformal invariance, to see that
 $c_1 \leq |\hat z|, |\hat w| \leq c_2  $. Therefore, \eqref{nov15}
 implies that
 \[     G^{t,t}(\hat z,\hat w) \leq c \, e^{-\beta/t} \, G(\hat z, \hat w) . \]
 Hence,
\begin{eqnarray*}
   \E[|g'(z)|^{2-d} \, |g'(w)|^{2-d}
  \, G^{t,t}(\hat z,\hat w) ; \xi \leq t, \eta = \eta_1]
    & \leq &  c \, e^{-\beta/t} \,
  \E[|g'(z)|^{2-d} \, |g'(w)|^{2-d}
  \, G(\hat z,\hat w) ]\\
  & \leq & c \, e^{-\beta/t} \, G(z,w) .
  \end{eqnarray*}
  The second inequality follows from the fact that $M_t(z,w)$ is
  a martingale.

  Now suppose $\eta = \eta_2$ and let $V$ be the component of
  $H  \setminus \eta$ containing $z$.  Since $\hcap[\gamma_\eta]
   \leq t$,  we can deduce \cite{LLN} that ${\rm Area}(V) \leq c \, t$.
   We claim that there exist $c,\alpha$ such that
\begin{equation}  \label{nov15.2}
        S(z) \leq c \,y \, r^{1/2} \, e^{-\alpha/t}.
        \end{equation}
To see this we use estimates on extremal distance  as in
the proof of Lemma \ref{aug31.lemma1}
to see that if $\zeta
   \in V$ with $|\zeta| \leq 1/2$, then the probability that a Brownian
   motion starting at $\zeta$ exits $V$ at $\eta$ is $O(e^{-\alpha/t})$
   for some $\alpha$.   However, combining the Beurling estimate
   and the gambler's ruin estimate, the probability that a Brownian
   motion starting at $z$ reaches the circle of radius $1/2$ without
   leaving $H$ is $O(r^{1/2} y)$.  This gives \eqref{nov15.2}
   and since $\Upsilon_H(z) \asymp  r \, y$,
   we get
   \[ |g'(z)|^{2-d} \, G(\hat z) = G_H(z; \gamma(\xi), \infty)
       \leq c \ (ry)^{d-2} \, (yr^{1/2} \, e^{-\alpha/t})^{4a-1}
        \leq c \, G(y) \, e^{-\beta/t} \, r^{ d-2} \, r^{  2a - \frac 12} . \]
   Hence,
    \[ \E\left[  |g'(z) \, g'(w)|^{2-d} \, G(\hat z,\hat w)
   \mid  2^{-n} \leq r \leq 2^{-n+1} , \eta = \eta_2 \right]
     \hspace{.5in} \]
     \[  \leq c \, e^{-\beta/t} \, 2^{-n(d-2) }\, 2^{-nu}   \,
      G(z) \, G(w) , \]
      where $u = 2a - \frac 12 >0$.
      Proposition \ref{mar14.prop1} implies that
  \[  \Prob\{
  2^{-n} \leq r \leq 2^{-n+1} , \eta = \eta_2 \} \leq  \Prob\left\{2^{-n} \leq r \leq 2^{-n+1} \right\}
      \leq c\, 2^{-n(2-d)}, \]
   and hence
    \[ \E\left[  |g'(z) \, g'(w)|^{2-d} \, G(\hat z,\hat w)
;  2^{-n} \leq r \leq 2^{-n+1} , \eta = \eta_2 \right]
   \leq c \, e^{-\beta/t} \, 2^{-un}
      G(z) \, G(w) . \]
   The estimate is obtained by summing over $n$.
%
%
%
%
%

\end{proof}

We will write $dA(z,w) = dA(z) \, dA(w)$ and
\[    \int_{V_1 \times V_2} f(z,w) \, dA(z,w) = \int_{V_1} \left[\int_{V_2} f(z,w)
 \, dA(w) \right] \, dA(z) , \]
 \[   \int f(z,w) \, dA(z,w) = \int_{\Half \times \Half} f(z,w) \, dA(z,w). \]

\begin{lemma}
\begin{equation}  \label{sept23.1}
 \int   G^{1,1}(z,w) \, dA(z,w) < \infty .
 \end{equation}
Moreover, there exists $c < \infty$ such that for every
$\epsilon > 0$,
\begin{equation}  \label{may14.5}
   \int
   G^{1,1}(z,w) \, 1\{|z-w| \leq \epsilon \}
   \, dA(z,w) \leq c \, \epsilon^d .
   \end{equation}
   \end{lemma}

\begin{proof} By symmetry it suffices to bound the
integrals on the set $\{|z| \leq |w| \}$.

If $z \in \Half$, and $|w| \geq 2 |z|$,
 then \eqref{may19.11} and \eqref{may19.12} imply  that
 there exist $\beta > 0$ such that
\begin{eqnarray*}
  \int_{|w| \geq 2 |z|} G^{1,1}(z,w) \, dA(w)
  & \leq & c\, |z|^{d-2} \, \int_{|w| \geq 2 |z|} \, e^{-\beta |w|^2}
    \, |w|^{d-2} \, dA(w) \\
    & \leq & c \, |z|^{d-2}\, \int_{2|z|}^\infty r^{d-1} \, e^{-\beta r^2}
     \, dr \\
      & \leq & c \,  |z|^{d-2} \, e^{-2\beta|z|^2}.
\end{eqnarray*}
Therefore,
\begin{equation}  \label{may14.1}
  \int G^{1,1}(z,w) \,1\{
   |w| \geq 2 |z|\} \,  dA(z,w)
  \leq  c \, \int |z|^{d-2} \, e^{-2 \beta |z|^2}
   \, dA(z) < \infty .
   \end{equation}

Let
\[   V_k = \{z \in \Half: 2^{-k-1} <|z| \leq 2^{-k} \}, \]
and let $\psi_m(z,w)$ be the indicator that $|z| \leq |w|$ and
$|z-w| \leq  2^{-m}$.  If $z \in V_0$ and $|z| \leq |w|$,
Theorem \ref{sept8.theorem} implies that
\[   G(z,w) \leq c \, |z-w|^{d-2}, \]
and hence
\[
 \int_{V_0  \times \Half}
G(z,w) \, \psi_{m}(z,w) \, dA(z,w)
     \leq   c \, \int_{|w| \leq 2^{-m}}
  |w|^{d-2} \,  dA(w) \leq  c \,   2^{-md}
    .\]
 More generally, if $k$ is an integer,  we can use
 the scaling rule for the Green's function to see that
\begin{eqnarray*}
      \int_{V_k \times \Half}
G(z,w) \, \psi_{m}(z,w) \, dA(z,w)  &
   = & 2^{2k(2-d)}  \int_{V_k \times \Half}
G(2^kz,2^kw) \, \psi_{m-k}(2^kz,2^kw) \, dA(z,w)
\\&=&
2^{-2kd}   \int_{V_k \times \Half}
G(2^kz,2^kw) \, \psi_{m-k}(2^kz,2^kw) \, dA(2^kz,2^kw) \\
& = & 2^{-2kd}  \int_{V_0 \times \Half}
G(z,w) \, \psi_{m-k}(z,w) \, dA(z,w) \\
& \leq &  c\, 2^{-kd}
  \, 2^{-md}.
\end{eqnarray*}

If $z \in V_k$ and $w \in V_k \cup V_{k-1}$, then
$|z-w| \leq 2^{-k} + 2^{1-k} < 2^{2-k}$.  Setting $m=k-2$
in the last inequality, we see that
\[ \int_{V_k \times (V_k \cup V_{k-1})}
    G(z,w) \, dA(z,w) \leq c \, 2^{-2kd}. \]
Combining this with  \eqref{may14.1}, we see
that
\[  \int_{|z|,|w| \leq 1} G(z,w) \,dA(z,w):=
C_0 < \infty . \]
The scaling rule gives
\[   \int_{|z|,|w| \leq 2^{k}} G(z,w) \,dA(z,w) =
C_0 \, 2^{2kd} . \]
Using \eqref{may19.12}, we get
\begin{eqnarray*} \lefteqn{
 \int_{|z|,|w| \leq 2^{k}} G^{1,1}(z,w) \, dA(z,w) }
\hspace{1in} \\
& \leq &
  \int_{|z|,|w| \leq 2^{k-1}} G^{1,1}(z,w) \, dA(z,w)
   + c \, e^{- \beta 2^k} \,
    \int_{|z|,|w| \leq 2^{k}} G(z,w) \,dA(z,w)
    \\
    & \leq &
 \int_{|z|,|w| \leq 2^{k-1}} G^{1,1}(z,w) \,dA(z,w)
  + c \,  e^{- \beta 2^k} \, C_0 \, 2^{2kd}.
  \end{eqnarray*}
By summing over $k \in \Z$, we see that
\[  \int G^{1,1}(z,w) \, dA(z,w) < \infty.\]

Using \eqref{may19.12} again, we get
\begin{eqnarray*}
 \int_{V_k  \times \Half}
G^{1,1}(z,w) \, \psi_{m}(z,w) \, dA(z,w)
&  \leq  & c \, e^{- \beta \, 2^{-2k}}
      \int_{V_k  \times \Half}
G(z,w) \, \psi_{m}(z,w) \, dA(z,w)\\
& \leq & c \, 2^{-md} \, e^{-\beta \, 2^{-2k}} \,
  2^{-kd} .
  \end{eqnarray*}
  By summing over   $k \in \Z$, we get
\[
 \int
G^{1,1}(z,w) \, \psi_{m}(z,w) \, dA(z,w)
  \leq c \, 2^{-md} \]
  which gives the second estimate.
\end{proof}

\section{Natural parametrization}

\subsection{Natural parametrization in $\Half$}  \label{nathalfsec}

Recall that
\[  L(s,t) = \int M_s^t(z) \, dA(z) \]
and if $r < s < t$,
\begin{eqnarray*}
 \E\left[L(r,t) - L(s,t) \mid \F_r\right]& = &
    \int M_r^s(z) \, dA(z) \\\
    & = &
    L(r,s) =
    \int |f_r'(z)|^d \, G^{s-r}(Z_r(z)) \, dA(z) .
    \end{eqnarray*}
 If we fix $t_0$, then
 \[    L(t,t_0) , \;\;\; 0 \leq t \leq t_0 \]
 is a positive supermartingale.  The {\em natural parametrization
 or natural length} $\Theta_t$ is the unique adapted increasing
 process $\Theta_t$ such that for each $t_0$,
 \[           L(t,t_0)  + \Theta_t ,\;\;\;\;\;  0 \leq t \leq t_0 \]
 is a martingale.   Although it may appear at first glance
 that the
 definition of $\Theta_t$ depends on $t_0$,  the
 fact that
 $N_t$ in \eqref{Diffmartingale} is martingale
 implies that   $\Theta_t$ is independent of $t_0$.  In
 our proof of the existence, it is shown
 that    with probability one
 for all $0 < s < t < \infty$,
\begin{equation}  \label{oct21.3}
  \Theta_t - \Theta_s
 = \lim_{n \rightarrow \infty} [\Theta_t^{(n)}
 - \Theta_s^{(n)}] ,
 \end{equation}
 where
\[     \Theta_t^{(n)} = \sum_{j < t2^{n}}
      L\left(\frac{j}{2^n}, \frac{j+1}{2^n} \right). \]
  Our definition here is slightly different but equivalent
  to the one used in \cite{LS,LZ}.  The equivalence can be seen
  in that in both cases we derive the formula \eqref{oct21.3}.
  We note that if $t = k \, 2^{-n}$ for positive integer $k$,
  then
\begin{equation}  \label{oct21.5}
\E\left[\Theta_t^{(n)}\right] = \E[L(0,t)]
 = \int G^t(z) \, dA(z) = t^d \, \int G^1(z) \, dA(z)
    < \infty .
   \end{equation}

The existence of such a process was established in \cite{LS}
for $\kappa < 5.0\cdots$ and for all $\kappa < 8$ in \cite{LZ}.
In this section, we reprove the results with some
improvements.   We will consider only the case $t_0 = 1$;
other values of $t_0$ can be handled using scaling.  For
the remainder of this section, we let
\[               L_t = L(t,1)    , \]
and recall that $L_t = 0$ if $t \geq 1$.  Proof of existence
follows from the Doob-Meyer decomposition once we have obtained
sufficient bounds on the second moment.
The key is the following proposition
which is very similar to  \cite[(16)]{LZ}. The proof is quick because
the hard work  is in
 the  estimate  \eqref{twopointer} on the time-dependent
Green's function.

\begin{proposition}  There
exists $C < \infty$ such that
for every stopping time $T \leq 1$,
\begin{equation}  \label{sept23.2}
      \E\left[L_T^2\right] \leq C.
    \end{equation}
\end{proposition}

\begin{proof}  Let  $u$   denote  the constant $c$ appearing
in \eqref{twopointer}.  Let $G^{t}(z,w)  =
 G^{t,t}(z,w)$.  We   establish the estimate with
\[   C =   u\int G^{u}(z,w)
\, dA(z,w)  \]
which is finite by  \eqref{sept23.1}.
  Let $T$ be a stopping time and let
$g = g_T, Z = Z_T$.
Using the two-point martingale and \eqref{twopointer}, we
see that for all $z,w \in \Half$,
\begin{eqnarray*}
\E\left[M_T^{1}(z) \, M_T^{1}(w)\right] & = &
 \E\left[|g'(z) \, g'(w)|^{2-d} \,
G^{1 - T}(Z(z)) \, G^{1- T}(Z(w)) \right]  \\
 & \leq & u\,  \E\left[|g'(z) \, g'(w)|^{2-d} \,
G^{u(1 - T)}(Z(z), Z(w)) \right]  \\
& \leq & u\, G^{u(1 - T)}(z,w)\\
& \leq & u\, G^{u}(z,w) .
\end{eqnarray*}
Therefore,
\[
\E[L_T^2]   =   \int
 \E\left[M_T^{1}(z) \, M_T^{1}(w)\right] \, dA(z,w)
 \leq    u\int G^{u}(z,w)
\, dA(z,w)   .
\]

\end{proof}

With this estimate, we can see that the supermartingale
$L_t$ is in the ``class $DL$''  and a process $\Theta_t$ exists
which makes $L_t  + \Theta_t$ a martingale.
(See \cite[Section 1.4]{KS} for relevant facts about the
Doob-Meyer decomposition.)
Moreover,  we can
deduce \eqref{oct21.3} if we weaken our notion of limit
to weak-$L^1$ convergence, that
 is,  for every integrable $Y$
 \[       \lim_{n \rightarrow \infty} \E\left[(\Theta_1^{(n)} - \Theta_1)
  Y\right] = 0 . \]
For the restricted range of $\kappa$, an almost sure limit was
established as well as an estimate of the H\"older continuity
of the paths in \cite{LS}.
  One cannot conclude  these stronger results  using only the estimate
\eqref{sept23.2}.  In this section, we will establish the almost
sure limit and H\"older continuity for all $\kappa < 8$ by
giving a stronger moment bound. Given the existence of $\Theta_t$,
we can also write
\[  L(s,t) = \E\left[\Theta_t - \Theta_s \mid \F_s\right]
 = \int_\Half |f_s'(z)|^{d} \, G^{t-s}(z) \, dA(z) . \]

Before proving this theorem. we make a comment that will
be important when we study the natural parametrization in
a smaller domain.  Suppose $\{\sigma_s\}$ is a collection of
bounded
stopping times for the supermartingale $L_t$ such that with
probability one $s \mapsto \sigma_s$ is continuous and
strictly increasing.
Then,
\[          \tilde L_s = L_{\sigma_s} , \]
is a time change of a positive supermartingale and hence is
a positive
supermartingale.   A stopping time $T$ for $\tilde L_s$ can
be considered as a stopping time for $L_t$ and hence we
can see that
$\E[\tilde L_s^2]$ is uniformly bounded.  We can apply the
Doob-Meyer theory to $\tilde L_s$ and conclude that there
is an increasing process $\tilde \Theta_s$ such that
$\tilde L_s + \tilde \Theta_s$ is a martingale.  Indeed,
it is easy to see that
$\tilde \Theta_{s} = \Theta_{\sigma_s}$.  If $\tilde \F_s=
\F_{\sigma_s}$ and
\[   \tilde \Theta_s^{(n)}  =  \sum_{0 \leq j <s 2^n}
   \E\left[\tilde L_{j2^{-n}} -\tilde L_{(j+1)2^{-n} } \mid \tilde\F_{j2^{-n}}
    \right] , \]
    then we know that $\tilde \Theta_s^{(n)} \rightarrow
    \tilde \Theta_s = \Theta_{\sigma_s}$, at least weakly in
    $L^1$.   It is difficult, in general, to give useful
    expressions for
   \[   \E\left[\tilde L_{j2^{-n}} -\tilde L_{(j+1)2^{-n} } \mid \tilde\F_{j2^{-n}}
    \right] . \]
 However, we will use this with a smooth reparameterization for which
we can estimate   the conditional expectation.

   Our proof  of almost sure convergence of $\Theta_t^{(n)}-
   \Theta_s^{(n)} $
   will
 make use of the H\"older continuity of $SLE_\kappa$  curves
 with $\kappa \neq 8$
 which we now review.
  We say that a function $f:[0,\infty) \rightarrow \R$
 is {\em weakly $\alpha$-H\"older} if
 it is continuous at $0$ and for every
 $0 < s_1 < s_2 < \infty$ and $\beta < \alpha$,
 $f(t), s_1 \leq t \leq s_2$  is  H\"older continuous of
 order $\beta$.
The following was proved in \cite{Lind}; see also \cite{LJ1}
where $\alpha_*$ is shown to be optimal.

 \begin{lemma}  \label{holder}
 Suppose that $\kappa > 0, \kappa \neq 8$, and let
 \[ \alpha_* = \alpha_*(\kappa) =
 1 - \frac{ \kappa} {24 + 2 \kappa -
 8 \sqrt {8 + \kappa}}    > 0. \]
 Then with probability one, the $SLE_\kappa$ curve
 parameterized by half-plane capacity
 is weakly $\alpha_*$-H\"older.
 \end{lemma}

 We now state the main theorem in this section.

 \begin{theorem} \label{existtheorem2}
  If $\kappa < 8$, then with probability one
 for all $0 < s < t< \infty$ the limit
 \begin{equation}  \label{apr23.1}
                \Theta_t - \Theta_s= \lim_{n \rightarrow \infty}
      \left[\Theta_t^{(n)} - \Theta_s^{(n)} \right]
      \end{equation}
  exists and $t \mapsto \Theta_t$ is
  weakly $(\alpha_*d/2)$-H\"older.
  Moreover, for each $t$, $\Theta_t^{(n)} \rightarrow
  \Theta_t$ in $L^1$.
  \end{theorem}

 The bulk of the work in proving this theorem will come
 in establishing the next result.   We will  conclude
 this section by showing how to derive Theorem \ref{existtheorem2}
 from Theorem \ref{maytheorem2}.

\begin{theorem}  \label{maytheorem2}
Suppose $0 < \alpha < \alpha_*$ and $K <\infty$.  Let
$\tau_K = \tau_{K,\alpha}$ denote the stopping time
\[   \tau_K = \inf\left\{t: \exists s \in [0,t) \mbox{ with }
|\gamma(t) - \gamma(s)| \geq K(t-s)^\alpha \right\}. \]
Then if $1  \leq t  \leq 2$,  the limit
\[         \lim_{n \rightarrow \infty} \left[\Theta_{t \wedge \tau_K}^{(n)}  -
  \Theta_{1 \wedge \tau_K}^{(n)} \right] :=    Q_K(t)
  ,\]
 exists  in $L^2$ and with probability one.
 Moreover, with probability one, if $\beta < \alpha d/2$,
 the function $t \mapsto Q_K(t)$
 is  H\"older continuous of order $\beta$.
 \end{theorem}

 \begin{proof}  See Section  \ref{maysection}.
 \end{proof}

 \begin{proof}  [Proof of Theorem \ref{existtheorem2} given
 Theorem \ref{maytheorem2}]
 Suppose $\alpha < \alpha_*$, $\beta < \alpha d/2$,
  and let $\tau_K
= \tau_{K,\alpha}$ be as in Theorem
\ref{maytheorem2}.
Lemma \ref{holder}
 implies that with probability one $\tau_K > 0$ for all
$K > 0$ and
\[    \lim_{K \rightarrow \infty} \tau_K = \infty . \]
 Therefore, if $Q(t) = \lim_{K \rightarrow \infty} Q_K(t)$,
 with probability one, for $1 \leq t \leq 2$,
 \[         \lim_{n \rightarrow \infty} \left[\Theta_{t}^{(n)}  -
  \Theta_{1}^{(n)} \right] =    Q(t)
  .\]
Moreover, $t \mapsto Q(t)$ is H\"older continuous of order
$\beta$ on $[1,2]$.

 Using $SLE$ scaling, we can see that Theorem
 \ref{maytheorem2} implies that with probability one for
 every positive integer $m$ and $2^{-m } \leq t $,
 the limit
 \[    Q(t;m) =    \lim_{n \rightarrow \infty} \left[\Theta_{t }^{(n)}  -
  \Theta_{2^{-m}}^{(n)}   \right]
 \]
 exists. Moreover, with probability one, for $t_0 < \infty$,
 $t \mapsto Q(t;m) $
 is  H\"older continuous of order $\beta$ on $[2^{-m},t_0]$.
   Using
  \eqref{oct21.5} and
  Fatou's lemma, we get
  $\E[Q(t;m)] \leq c \, t^d$.  From this we can define
  by monotonicity
  \[   \Theta_t = \lim_{m \rightarrow \infty} Q(t;m) \]
  and we can see that $\E[\Theta_t] \leq c \, t^d$.
  Note that if $2^{-m} \leq s \leq t$, then
  \[           \Theta_t - \Theta_s = Q(t;m) - Q(s;m) . \]
 We claim  that with probability one, $\Theta_{0+}=0$.
Indeed,
 \[   \Prob\{\Theta_t > \epsilon\} \leq  \epsilon^{-1} \,
   \E[\Theta_t] \leq c \, \epsilon^{-1} \, t^d , \]
   and hence
  \[          \sum_{j=1}^\infty \Prob\{\Theta_{2^{-j}} > 2^{-dj/2}\}
     < \infty. \]
Therefore, by Borel-Cantelli, we get with probability one $\Theta_{0+} = 0.$

We have shown that $\Theta_t$ is weakly $\beta$-H\"older
for every  $\beta < \alpha_*d/2$; hence it is weakly
$(\alpha_*d/2)$-H\"older. \end{proof}

%
%
%
%

\subsection{Proof of Theorem \ref{maytheorem2}}  \label{maysection}

We fix $\alpha < \alpha_*$ and allow all constants to depend
on $\alpha$.
As before, let
 \[        M_t (z) = |g_t'(z)|^{2-d} \, G^{ }
 (Z_t(z)) , \]
 and if $ 0 \leq s < t \leq \theta$,
 \[       \Psi_t  = \int_\Half M_t (z) \, dA(z) , \]
 \[      L(s,t;K) = \E\left[\Psi_{s \wedge \tau_K}
  - \Psi_{t \wedge \tau_K}  \mid \F_{s} \right].\]
  The theorem is an example of the Doob-Meyer decomposition.
  To prove the existence of the decomposition there is an easy
  step and a hard step: the easy step is to discretize time and
  give an approximation and the hard step is to take the
  limit of the approximations.  This latter step is not so
  difficult if one establishes extra moment bounds.  In particular,
  our 
  theorem follows from \cite[Proposition 4.1]{LS} once
  we obtain the following estimate which we derive
  in this section.

  \begin{proposition} \label{nov18.prop2}
    For every $\epsilon >0$ and $K < \infty$,
there exists $c < \infty$ such that if $\epsilon \leq s < t
\leq 1/\epsilon$,
\[    \E\left[L(s,t;K)^2\right]\leq c \, (t-s)^{1+\alpha d}.\]
\end{proposition}
%
%
   In \cite{LS},
 a similar estimate
  was obtained for $\kappa < 5.0\cdots$
 using the reverse Loewner flow.  No such estimate was  established in \cite{LZ}
 for $\kappa < 8$,  and this is why the result there was not as strong.  The
 proof we give here uses the usual (forward) Loewner flow,
 as well as   the
 H\"older continuity of the $SLE$ curve (in the capacity parametrization).

%
 Let
 \[        M_t^\theta(z) = |g_t'(z)|^{2-d} \, G^{\theta - t}
 (Z_t(z)) , \;\;\;\;
      \Psi_t^\theta = \int M_t^\theta(z) \, dA(z) . \]
 For every $0 \leq s \leq t \leq \theta$,
 \[      L(s,t;K) = \E\left[\Psi_{s \wedge \tau_K}
 ^\theta - \Psi_{t \wedge \tau_K}^\theta \mid \F_{s} \right],\]
 and, in particular,
 \[    L(s,t;K) = \E\left[\Psi_{s \wedge \tau_K}
 ^t - \Psi_{t \wedge \tau_K}^t \mid \F_{s} \right].\]
 Note that for fixed $s$, $L(s,t;K)$ is  increasing
 in $t,K$.

\begin{lemma}  \label{lscaling}
If $r > 0$ and $0 \leq s < t$,
then the distribution of $L(s,t;K)$ is the same as that
of $r^{-d} \,L(r^2s,r^2t; r^{2\alpha - 1}K)$.
In particular, it has the same distribution as
\[        s^{d/2} \, L(1,t/s;s^{\frac 12 - \alpha}K).\]
\end{lemma}

\begin{proof}
We use the scaling rule for $SLE_\kappa$.
Fix $r$ and let
\[  \tilde g_t(z)   = r^{-1} \, g_{r^2t}(rz),\;\;\;\;
     \tilde \gamma(t)
      = r^{-1} \, \gamma(r^2t).\]
Then
$\tilde g_{t},
\tilde \gamma$ have the same distribution as
$g_t,\gamma$.  Indeed,$\tilde g_{t}$   is the solution to the Loewner
equation with driving function $\tilde U_t = r^{-1}
\, U_{r^2t}$.  Let
\[  \tilde Z_t(z)  = \tilde g_t(z) -
  \tilde U_t = r^{-1} \, Z_{r^2t}(rz). \]
   We define
\[   \tilde \tau_K  = \inf\{t: \exists s \leq t:
  |\tilde \gamma(t) - \tilde \gamma(s)|  = K \, (t-s)^\alpha\}.\]
  Since
  \[
 \frac{|\tilde \gamma(t) - \tilde \gamma(s)|
} {(t-s)^\alpha} =
     \frac{|\gamma(r^2t) - \gamma(r^2s)|}
    {r \, (t-s)^\alpha} =
        \frac{|\gamma(r^2t) - \gamma(r^2s)|}
    {r^{1-2\alpha} \, (r^2t-r^2s)^\alpha},\]
 we get
 \[    \tilde \tau_K = r^{-2} \, \tau_{r^{2\alpha -1}K}.\]
 Let  $\tilde M_t^\theta(z), \tilde \Psi_t^\theta$
 be defined as above using $\tilde g, \tilde \gamma$.
  Note that
 \begin{eqnarray*}
 \tilde M_t^\theta(z) & = & |\tilde g_t'(z)|^{2-d}
   \, G^{\theta - t}(\tilde Z_t(z))\\
    & = & |g_{r^2t}'(rz)|^{2-d} \, G^{\theta - t}(Z_{r^2t}(rz)/r)\\
    & = & |g_{r^2t}'(rz)|^{2-d} \,r^{2-d} \,
     G^{r^2\theta - r^2t}(Z_{r^2t}(rz))\\
     & = & r^{2-d} \, M_{r^2t}^{r^2\theta}(rz),
   \end{eqnarray*}
 \begin{eqnarray*}
 \tilde \Psi_t^\theta & = & \int \tilde M_t^\theta(z) \, dA(z)\\
  & = & r^{2-d}  \int  M_{r^2t}^{r^2\theta}(rz) \, dA(z)\\
& = & r^{-d}  \int   M_{r^2t}^{r^2\theta}(rz) \, dA(rz)
  = r^{-d} \, \Psi_{r^2t}^{r^2\theta}. \end{eqnarray*}
 Moreover,  for each $K$,
  \[   \tilde \Psi_{t \wedge \tilde \tau_K}^\theta
             = r^{-d} \, \Psi_{r^2t\wedge \tau_{r^{2\alpha - 1}K
             }}^{r^2\theta}
               .\]
Since $\tilde \F_t = \F_{r^2t}$ we can take conditional
expectations and get the result.
 \end{proof}

 \begin{corollary} \label{lscaling2}
 If $n$ is a positive integer,
 \[  \E\left[L(1,1+2^{-n-1};K)^2\right] \leq
         2^{-n} \sum_{j=1}^{2^n} \E\left[L(1 + (j-1)2^{-n}, 1 + j2^{-n};
             K)^2 \right]. \]
 \end{corollary}

\begin{proof}  Letting $s = 1 + (j-1)2^{-n} \in [1,2]$, we have
\begin{eqnarray*}
 \E\left[ L(1 + (j-1)2^{-n}, 1 + j2^{-n};
             K)^2\right
             ]  & =&  s^{d} \,\E\left[ L\left(1, 1 + \frac{1}{2^n +j-1}
             ; s^{\frac 12 - \alpha}K \right)\right]^2\\
             &
             \geq & \E\left[L( 1 , 1 + 2^{-(n+1)}; K)^2\right].
\end{eqnarray*}
             \end{proof}

%

\begin{proof} [Proof of Proposition \ref
{nov18.prop2}]
Using Lemma \ref{lscaling} we see that it  suffices
to prove the result when $s=1$ and $t-s = 2^{-n}$ for some positive
integer $n$.  By Corollary
\ref{lscaling2}, it suffices to prove that for all $K < \infty$,
\[   \sum_{j=1}^{2^n} \E\left[L(j,n)^2\right]
 \leq c_K \, 2^{-n \alpha d}, \]
 where
\[    L(j,n) = L\left(1 + (j-1)2^{-n}, 1 + j2^{-n}; K \right).\]
For the remainder, we fix $K$ and allow constants to depend on
$K$.    We write $G^{s}(z,w)$ for $G^{s,s}(z,w)$.

  Let $\tau = \tau_K$ and
\[      Y_t = \Psi_{(1+t) \wedge \tau} . \]
The definition of $\tau$ implies that
 $|\gamma((t+s) \wedge \tau) - \gamma(t \wedge \tau)|
 \leq K s^{\alpha}$.   Therefore,
 \[  \E\left[M_{(t+s) \wedge \tau }(z) \, 1\{|z - \gamma(t)| > Ks^\alpha\}
  \right] = M_{t \wedge \tau}(z) \, 1\{|z-\gamma(t)| >Ks^\alpha \}. \]
In particular, if
 $t = 1+(j-1)2^{-n}$,
\[
 \E\left[Y_{ (j-1)2^{-n}} - Y_
 {j2^{-n}}    \mid \F_t \right]
 \leq \int |g_t'(z)|^{2-d}
   \, G^{2^{-n}}(Z_t(z)) \, 1\left\{|z- \gamma(t)|
     \leq K \, 2^{-\alpha n} \right\}\, dA(z).
\]
Hence,  if
 \[    J_n = \{(z,w) \in \Half^2: |z-w| \leq 2K2^{- \alpha n}\},\]
 then
 \[  L(j,n)^2 \leq \int_{J_n} |g_t'(z)|^{2-d}
 \, |g_t'(w)|^{2-d} \,
   \, G^{2^{-n}}(Z_t(z)) \, G^{2^{-n}}(Z_t(w))
    \,   dA(z,w) .\]
Lemma \ref{may14.lemma1} implies that there exists $u < \infty$
such that
\[    L(j,n)^2 \leq u \int_{J_n} |g_t'(z)|^{2-d}
 \, |g_t'(w)|^{2-d} \,
   \, G^{u2^{-n}}(Z_t(z),Z_t(w))
    \,   dA(z,w) .\]
Without loss of generality we assume $u = 2^{m_0}$ for some integer
$m_0 \geq 3$.

Let
\[         \hat M_t(z,w) =   |g_t'(z)|^{2-d}
 \, |g_t'(w)|^{2-d} \,
   \,  G^{u+1-t}(Z_t(z),Z_t(w)), \]
and note that if $t=j2^{-n}$,
\[    \E\left[   \hat  M_t(z,w) -   \hat  M_{t + u2^{-n}}(z,w)
   \right]
    \geq \E\left[|g_t'(z)|^{2-d}
 \, |g_t'(w)|^{2-d} \,
   \,  G^{u2^{-n}}(Z_t(z),Z_t(w))\right].\]
 (Roughly speaking, the right-hand side corresponds to the
 normalized  probability
 that the path goes through both $z$ and $w$  between times
 $t$ and $t + u2^{-n}$.)
By summing over $j$, we see that
\[  \sum_{j=1}^{2^{n-m_0}} \E[L(jm_0,n)^2]
  \leq c \int_{J_n} G^{u+1}(z,w)
    \, dA(z,w). \]
    Similarly, for $k=0,1,2,\ldots,m_0-1$, we see that
    \[  \sum_{j=1}^{2^{n-m_0}} \E[L(jm_0 + k ,n)^2]
  \leq c \int_{J_n} G^{u+k2^{-n} +1}(z,w)
    \, dA(z,w). \]
By summing over $k$, we see that
  \[  \sum_{j=1}^{2^{n}} \E[L(j,n)^2]
  \leq c \int_{J_n} G^{2u}(z,w)
    \, dA(z,w). \]
    and by using
 \eqref{may14.5} and scaling
we see that
\[ \sum_{j=0}^{2^n-1} \E[L(j,n)^2]
  \leq c \, 2^{-d\alpha n}.\]

 \end{proof}

The next lemma will be used later.

\begin{lemma}  \label{laterlemma}
 There exists $c < \infty$ such that
if $s,t > 0$ and $r \leq 1/2$, then
\[        L(s,s+(1+r)t) - L(s,s+t) \leq c \, r \, L(s,s+t).\]
\end{lemma}

\begin{proof}  Let $f(z) = g_s^{-1}(U_s + z)$
and $u = \sqrt{1+r}$.
   Then,
\[    L(s,s+ t) = \int_\Half |f'(z)|^d \, G^t(z) \, dA(z) .\]
\begin{eqnarray*}
    L(s,s+ (1+r)t)&   = &   \int_\Half |f'(w)|^d \, G^{t(1+r)}(w) \, dA(w) \\
    & = & u^{d -2}
    \int_\Half |f'(w)|^d \, G^t(w/u) \, dA(w)  \\
    & = & u^{d} \int_\Half |f'(uz)|^d \, G^t(z) \, dA(z)
\end{eqnarray*}
The distortion theorem implies that
\[  |f'(uz)| = |f'(z)| \, [1 + O(u-1)] = |f'(z)| \, [1 + O(r)].\]
Since $u^d = 1+O(r)$,
\[ L(s,s+(1+r)t) = [1 + O(r)]
\int_\Half |f'(uz)|^d \, G^t(z) \, dA(z) =
 L(s,s+t) \,   [1 + O(r)]   .\]
 \end{proof}

%

\subsection{Two useful lemmas}  \label{usefulsec}
We have shown that
\begin{equation}  \label{oct10.1}
      \Theta_t =\lim_{n \rightarrow \infty}
  \sum_{j \leq t2^n}
  \int_\Half |f_{(j-1)2^{-n}}'(z)|^d \,
       G^{2^{-n}}(z) \, dA(z) ,
       \end{equation}
       where the limit is in $L^1$.
The integral is concentrated on $z$ near the origin.
Proposition \ref{prop.oct9}
makes a precise statement of this.   Recall
L\'evy's theorem on
the modulus of continuity of Brownian motion (see, e.g.,
\cite[Theorem 9.5]{KS}) which states that
 if $U_t$ is a standard Brownian motion and $t_0 < \infty$,
then with probability one
\begin{equation}  \label{modulus}
      \sup\left\{\frac{|U_t - U_s|}{
   \sqrt{|t-s| \, |\log (t-s)|}} : 0 \leq s < t \leq t_0\right\}
   < \infty.
   \end{equation}

\begin{proposition}   \label{prop.oct9} Suppose $\delta_n$ is a sequence
of positive numbers such that
\[     \lim_{n \rightarrow \infty} 2^{-n/2} \, \sqrt n \,
  \delta_n^{-1} = 0 . \]
Let $\Half_n = \{z \in \Half: |z| \leq \delta_n\}$.
Then,
\[ \Theta_t =\lim_{n \rightarrow \infty}
  \sum_{j \leq t2^n}
  \int_{\Half_n}   |f_{(j-1)2^{-n}}'(z)|^d \,
       G^{2^{-n}}(z) \, dA(z) , \]
  where the limit is in $L^1$.
  \end{proposition}

\begin{proof} Let $\epsilon_n$ be a sequence with
\[         \lim_{n \rightarrow \infty} 2^{-n/2} \, \sqrt n \,
  \epsilon_n^{-1} = 0 , \]
  \[     \lim_{n \rightarrow \infty} \epsilon_n \, \delta_n^{-1} = 0.\]
Let
\[         \gamma^{(s)}(t) =g_s(\gamma(t+s)) -
g_s(\gamma(s)) =  g_s(\gamma(t+s)) -
  U_s , \]
and let
\[    \lambda_n = \inf\{t: \exists m \geq n, s \leq t \mbox{ with }
|t-s| \leq 2^{-m} \mbox{ and } |\gamma^{(s)}(t)|
   \geq \delta_m\}, \]
   \[  \hat \lambda_n =  \inf\{t: \exists m \geq n , s \leq t \mbox{ with }
|t-s| \leq 2^{-m} \mbox{ and } |U_t - U_s|
   \geq  \epsilon_m\}.  \]
Deterministic estimates using the Loewner equation
(see \cite[Lemma 4.13]{Law1}) show that
for all $n$ sufficiently large, $\hat \lambda_n \leq \lambda_n$.
Using \eqref{modulus} can see that with probability one $\hat \lambda_n
\uparrow \infty$ and hence
  $\lambda_n\uparrow
\infty$.

%
 Note that for fixed $m$, if $n \geq m$,
 \[ \E[\Theta_{(s +2^{-n})  \wedge \lambda_m} -
    \Theta_{s \wedge \lambda_m}
     \mid \F_{s \wedge \lambda_m}]
       \leq
\int_{\Half_n}   |f_{s}'(z)|^d \,
       G^{2^{-n}}(z) \, dA(z) .\]
Therefore,
\[    \Theta_{t \wedge \lambda_n} \leq
        \liminf_{n \rightarrow \infty}
          \sum_{j \leq t2^n}
  \int_{\Half_n}   |f_{(j-1)2^{-n}}'(z)|^d \,
       G^{2^{-n}}(z) \, dA(z) .\]
However $\Theta_{t \wedge \lambda_n}$ converges
monotonically to $\Theta_t$ and hence $\E[\Theta_t
 - \Theta_{t \wedge \lambda_n} ] \rightarrow 0$.
 Combining with \eqref{oct10.1}, we get the result.
 \end{proof}

 \begin{lemma}  \label{uniformlemma}
  Suppose $U_t$ is a driving function such that the
Loewner equation generates a curve.  Then for all $T < \infty$, the
functions $\{f_s(z): 0 \leq s \leq T\}$ are equicontinuous at $0$.
In other words, for every $\epsilon >0$, there exists $\delta > 0$
such that if $|z| < \delta$ and $t \in [0,T]$, then $|f_t(z) - \gamma(t)|
< \epsilon$.
\end{lemma}

\begin{proof} if $s < t$, define $g_s^t$ by $g_t = g_s^t \circ g_s$
and write $h_t = g_t^{-1}, h_s^t = (g_s^t)^{-1}$. Note that if
$s<t$, then $h_s = h_t \circ g_s^t$ and $h_t = h_s \circ h_s^t$.

Since $U_t$  is uniformly continuous on $[0,T]$, there exists increasing
$v(s)$
with $v(0+)  = 0$ such that
$  |U_{t+s} - U_t| \leq v(s), 0 \leq t \leq T-s$.  From the Loewner
equation (see \cite[Lemma 4.13]{Law1}), one can see that there
exists a universal $c < \infty$  such that for all $z$,
and $ 0 \leq t \leq T-s$,
\begin{equation}  \label{oct21.1}
            |g_{t}^{t+s}(z) - z|  \leq  c \, [v(s) \vee \sqrt s] , \;\;\;\;
   |h_t^{t+s}(z) - z|
   \leq  c \, [v(s) \vee \sqrt s].
   \end{equation}

Since the Loewner equation is generated by a curve, we know
that for each $t$, $f_t(z)$ can be extended
continuously to $\overline{\Half}$  with $f_t(0) = h_t(U_t)
= \gamma(t).$ We claim the following stronger fact: for every
$\epsilon > 0$ and $t \in [0,T]$,
 there exists $\delta = \delta(t,\epsilon)  > 0$ such that if $|z|< \delta$
and $|t-s| < \delta$, then $|f_s(z) - \gamma(t)| < \epsilon$.  To see this,
fix $t$ and find $r$ such that
$|f_t(z) - f_t(0)| < \epsilon/2$ for $|z| < r$.
Using \eqref{oct21.1}, we can find
 $\delta < r/4$ so that if $|s-t| < \delta $,
\[  |U_t - U_s| < r/4 \]
and
\[    \sup_w  |g_s^t(w) - w| \leq r/4 \;\;\;\; \mbox{ if } \;\;\; s < t , \]
\[   \sup_w |h_t^s(w) - w | \leq r/4 \;\;\;\;  \mbox{ if } \;\;\;
 s > t . \]
 Then if $t-\delta < s < t$ and $|z| < \delta$, we have
\[  |f_s(z) - f_t(0)| = |h_s(z + U_s) - f_t(0)| =
  |h_t(g_s^t(z+U_s)) - f_t(0)| = |f_t(w) - f_t(0)|, \]
  for some $w$ with $|w| < r$.  Hence $|f_s(z) - f_t(0)|
  < \epsilon/2$.   A similar argument proves this estimate
  for $t < s < t+\delta$.
Therefore, if $|z| < \delta$ and
  $|s-t| < \delta$,
  \[ |f_s(z) - \gamma(t)| =   |f_s(z) - f_t(0)| \leq
    |f_s(z) - f_s(0)| + |f_s(0) - f_t(0)| <   \epsilon . \]

  We finish the proof by a standard compactness argument, covering $[0,T]$
  by a finite number of intervals $[t - \delta_t, t+ \delta_t]$ and then choosing
  $\delta = \min\{\delta_t\}$.
\end{proof}

 \subsection{Natural parametrization in other domains}  \label{natothersec}

 Suppose $D$ is a simply connected domain with distinct boundary
 points $z,w$.  Then $SLE_\kappa$ from $z$ to $w$ in $D$ is defined
 (up to reparameterization) by $\eta(t) = F \circ \gamma(t)$, where
 $\gamma$ is an $SLE_\kappa$ in $\Half$ from $0$ to $\infty$ and
 $F: \Half \rightarrow D$ is a conformal transformation with $F(0) = z,
 F(\infty) = w$.  The conformal transformation is unique only up to an
 initial dilation, but scaling shows that the distribution is independent
 of the choice.    If $\Theta_t$ denotes the natural parametrization for
 the curve $\gamma$, we {\em define} the natural parametrization
 $\tilde \Theta_t$ in
 $D$ by the scaling rule
\begin{equation}  \label{otherdef}
        \hat \Theta_t = \int_0^t |F'(\gamma(s))|^d \, d \Theta_s.
        \end{equation}
  A simple argument shows that the joint distribution of $(F \circ \gamma
  (t), \hat \Theta_t)$, modulo time reparameterization,
   is independent of the choice of $F$.

   This definition can also be given more intrinsically.  Suppose
   for ease that $D$ is a bounded domain with
\begin{equation}  \label{oct21.6}
     \int_D G_D(\zeta;z,w) \, dA(\zeta) < \infty.
     \end{equation}
   Let $\eta(t)$ be an $SLE_\kappa$ path from $z$ to $w$ in $D$;
   the choice of parametrization is not important.  Then the
   natural length in $D$ is the unique increasing process
  $\tilde \Theta_t$ such that
  \[       \int_D G_{D_t}(\zeta; \eta(t),w) \, dA(\zeta)  +
      \tilde \Theta_t , \]
  is a martingale.  Here $D_t$ represents
     the component of
    $D \setminus \eta_t$ whose boundary includes $\eta(t)$
    and $w$.  In particular, $\eta$ has the natural parametrization
    if
    \[      \int_D G_{D_t}(\zeta; \eta(t),w) \, dA(\zeta)  +
   t  \]
   is a martingale.  If \eqref{oct21.6} does not hold we
   can define $\tilde \Theta$ as we did in the half plane
   by using some type of  cut-off.  (In this paper, the cut-off
   was given by using a time-dependent Green's function; in
   \cite{LS,LZ} the cut-off was made by considering the Green's
   function
 in a bounded domain bounded away from the origin.)

   If $\gamma[0,t]$ is a curve that lies in both $D_1$ and $D_2$, then
   it is not immediate from our definition that the natural length
   of $\gamma_t$ is independent of whether we consider it in $D_1$
 or in $D_2$. We will establish this in this
 section.  By conformal invariance, it
 suffices to prove this for $ D_1 = D = \Half \setminus V
  \subset \Half = D_2$ with $V$ bounded and $\dist(0,V) > 0$
  which we now assume.

  Let
 \[   F: \Half \longrightarrow D \]
 be the unique conformal transformation with
  $F(z) - z = o(1)$ as $z \rightarrow \infty$, and let $\Phi =
 F^{-1}$.  If $\gamma$ is an $SLE_\kappa$ path from $0$
 to $\infty$  in $\Half$,
 let
 \[     \tau_D = \inf\{t: \gamma(t) \in \overline V\}. \]
Suppose $\sigma$ is a stopping time with $\sigma < \tau_D$.
There are two probability measures on $\gamma_\sigma$,
that of
$SLE$ in $D$ and that of  $SLE$ in $\Half$ (in both cases
we are choosing endpoints $0$ and $\infty$).  It is known that these
two probability measures, considered as measures on the
stopped paths
$\gamma_\sigma$,
 are mutually absolutely continuous.
The natural length of $\gamma_\sigma$ can be considered using
either measure.  In the case of $SLE$ in $D$, one defines the
length in $\Half$ and then uses $F$ and the conformal covariance
rule \eqref{otherdef}
to get the length in $D$.  We will show the
two definitions are the same.  We start by setting
up some notation.

Let $\gamma(t)$ be an $SLE_\kappa$ path in $\Half$ defined
as usual with $\hcap[\gamma_t] = at.$   Let $g_t$
and $U_t$ denote the conformal maps and driving function,
respectively.  We assume $t < \tau_D$,
 and let

 \[     D_t = g_t(D), \;\;\;\;
       \eta^*(t) =   \Phi \circ \gamma(t) . \]
Let $H_t^*$ be the unbounded component of $\Half
 \setminus \eta^*_t$, and let
  $g^*_t$ denote the unique conformal transformation of
  $H^*_t$ onto $\Half$ with $g^*_t(z) - z = o(1)$
 as $z \rightarrow \infty$.
  Let
 \[    \Phi_t = g_t^* \circ \Phi \circ g_t^{-1} , \;\;\;\;
   F_t = \Phi_t^{-1} . \]
   Then $\Phi_t: D_t \rightarrow \Half$ is the unique
   conformal transformation with $\Phi_t(z) - z \rightarrow
   0$ as $z \rightarrow \infty$.
Let $\hat a(t)$ denote the capacity of the curve viewed
in $D$, that is,
\[           \hat a(t) = \hcap\left[\Phi \circ \gamma_t\right]
 = \hcap[\eta_t^*]. \]
Then \cite[Section 4.6]{Law1}
 \[  \hat a(t)  = a \int_0^t \Phi_s'(U_s)^2   \, ds, \]
 and
  $g^*_t$ satisfies the Loewner equation
   \[   \p_t g^*_t(z) = \frac{a \, \Phi_t'(U_t)^2}{
                    g_t^*(z) - U_t^*}, \]
   where $U_t^* = g_t^*(\eta^*(t)) = \Phi_t(U_t)$.  Let
   $\delta(t) = \dist(U_t,\Half \setminus D_t)$.  By Schwarz
   reflection, $\Phi_t$ can be extended to a conformal transformation
   on the open disk of radius $\delta(t)$ about $U_t$.  The distortion
   theorem implies that there exists a universal $c < \infty$ such that
    \[            |\Phi_t''(z)| \leq c \, \delta(t)^{-1} \, \Phi_t'(U_t),\;\;\;\;
               |z - U_t| \leq \delta(t)/2, \]
                \[            |\Phi_t'''(z)| \leq c \, \delta(t)^{-2} \, \Phi_t'(U_t),\;\;\;\;
               |z - U_t| \leq \delta(t)/2, \]
               \[                | \Phi_t'(z) - \Phi_t'(U_t)| \leq
      c \, \delta(t)^{-1} \, \Phi_t'(U_t)\, |z-U_t|,\;\;\;\;\;
         |z - U_t| \leq \delta(t)/2. \]
         Using the Loewner equation (see \cite[Proposition 4.41]{Law1}),
  we see that that there exist $c $ such that
  \[                \p_t \Phi_t'(x) \leq c \, \delta(t)^{-2}
   \, \Phi_t'(x) , \;\;\;\;\;  U_t - \frac{\delta(t)}{2} \leq
   x \leq U_t + \frac{\delta(t)}{2}. \]
  We will not need the full force of these estimates, but we
  will use the following consequence. Let $\hat \delta_t = \inf\{\delta(s):
 0 \leq s \leq t\}$, and
 \[  \Delta(s,t) = \max\{|U_r - U_s|: s \leq r \leq t\}.\]
 \begin{itemize}
 \item For every $\delta > 0$, there exists $\epsilon > 0, c < \infty$ such that
 if $\hat \delta_t   > \delta$, $0 \leq s \leq t \leq s+\epsilon$ and
 $\Delta(s,t) \leq \epsilon$, then
 \[           |\Phi_t'(U_t) - \Phi_s'(U_s)| \leq c\, [(t-s) + \Delta(s,t)]
        \, \Phi_s'(U_s). \]
  \end{itemize}

  Let $\sigma(t) = \inf\{s: \hat a(s) = at\}$, and let
  $\tilde \gamma(t) = \gamma(\sigma_t)$ denote the path $\gamma$
reparameterized so that $\hcap[ \Phi \circ \gamma_t] = at$. We
do this only  for
\[            t <  \tilde \tau_D = \inf\{s:
\tilde \gamma(s)  \in \overline V \}. \]
(Note that $\tau_D$ and $\tilde \tau_D$ are essentially the
same stopping time considered in two different parameterizations.)
For $t < \tilde \tau_D$, let
\[         \eta(t) = \Phi \circ \tilde \gamma(t)  = \eta^*(\sigma_t) , \]
and note that $\eta$ is parameterized so that
$\hcap[\eta_t] = at$.  Let $\hat g_t = g^*_{\sigma_t},
 \hat U_t = U_{\sigma_t}^*,
\hat \Phi_t = \Phi_{\sigma_t},\tilde U_t = U_{\sigma_t},
 \tilde g_t = g_{\sigma_t}, \hat f_t(z)=
\hat g_t^{-1}(z + \hat U_t)$, $\tilde f_t(z) = \tilde g_t^{-1}(z +
\tilde U_t)$, $\hat F_t = \tilde f_t^{-1} \circ F
  \circ \hat f_t.$   Note that
  \[             \tilde f_t \circ \hat F_t =
        F \circ \hat f_t.\]   Since
\[   at = \hcap[\eta_t] = a \int_0^{\sigma_t}  \Phi_s'(U_s^*)^2
 \,ds , \]
 we see that
 \[    \p_t \sigma_t = K_t \;\;\;\; \mbox{where}
 \;\;\;\;K_t =  \tilde \Phi_t'( \tilde U_t)^{-2} = \hat F_t'(0)^2. \]

Let $\Theta_t^*$ denote the natural parametrization associated
to the curve $\eta$ (under the measure of $SLE_\kappa$ from $0$
to $\infty$ in $\Half$).    Recall that
\begin{equation}  \label{sept25.1}
      \Theta_t^* = \lim_{n \rightarrow \infty}
    \sum_{j < t2^n}   L^*\left(\frac{j-1}{2^n}, \frac{j}{2^n}
      \right),
      \end{equation}
where
\[   L^*(r,s) = \int_\Half |\hat f_r'(z)|^d \, G^{s-r}(z) \, dA(z) ,  \]
and the limit is in $L^1$.
 Using Proposition \ref{prop.oct9}, we can also write
 \[  \Theta_t^* = \lim_{n \rightarrow \infty}
    \sum_{j < t2^n} J^*(j,n) \]
    where
  \[  J^*(j,n) =
     \int_{\Half_n}  | \hat f_{j-1,n}'(z)|^d \, G^{2^{-n}}(z) \, dA(z), \;\;\;\;
       \Half_n = \{z\in \Half: |z| \leq n 2^{-n/2} \}, \;\;\;
          \hat f_{j,n} = \hat f_{j2^{-n}} . \]
The expression on the right-hand side of \eqref{sept25.1}
 is a deterministic function of the
curve $\eta(s), 0 \leq s \leq t$.
Using continuity, it follows that
\begin{equation}  \label{oct9.3}
  \hat \Theta_t := \int_0^t |F'(\eta(s))|^d \,
  d\Theta^*_s =  \lim_{n \rightarrow \infty} \hat \Theta_{t,n}
  \end{equation}
   where
\[  \Theta_{t,n} =
    \sum_{j < t2^n} J^*(j,n)\, \left| F'\left(\eta\left( \frac{j-1}{2^n}
      \right) \right) \right|^{d}.
  \]
  The left-hand side of \eqref{oct9.3}
  is the natural parametrization of $\tilde \gamma$
considered as $SLE$ in $D$.   If $\Theta_t$ is the natural
parametrization associated to $\gamma$, then $\tilde \Theta_t
= \Theta_{\sigma_t}$ is the length of $\tilde \gamma$ considered
as an $SLE$ curve in $\Half$.   Having set up the notation,
we can now state the main theorem of this section.

\begin{theorem} \label{indyprop}
 With probability one, for all $t < \tau_D$,
$\tilde \Theta_t = \hat \Theta_t$.
\end{theorem}

By absolute continuity, the ``with probability one'' in the
statement can be taken either with respect to $SLE_\kappa$ in
$\Half$ or $SLE_\kappa$ in $D$.
Since both $\tilde \Theta_t$ and $\hat \Theta_t$ are continuous
in $t$ it suffices to show that for each $t < \tau_D$, $\tilde \Theta_t
= \hat \Theta_t$ with probability one.
To do this will we define a sequence of stopping times
$\tau_m$ with $\tau_m \uparrow \tau_D$ and prove that for
each $m$ and $t$
\begin{equation}  \label{oct22.1}
      \Prob \left\{ \tilde \Theta_{t \wedge \tau_m}\neq
 \hat \Theta_{t \wedge \tau_m}\right\} = 0.
 \end{equation}
   We let $\tau_m$ be the first
time $t$ such that one of the following occurs: $t \geq m$;
$\tilde \delta(t) \leq 1/m $;
  $\tilde \Phi_t'(\tilde U_t) \geq m$;  $\tilde \Phi_t'(\tilde U_t) \leq 1/m$;
 $|F'(\eta(t))| \geq m$, $|F'(\eta(t)| \leq 1/m.$  It
 is easy to check from our distortion estimates above
  that $\tau_m \uparrow \tau$.  For the remainder, we fix
 $ m$ and write $\tau = \tau_m$.   All constants, implicit or explicit,
 may depend on $ m$.  Also, we only need consider $t \leq m$,
 since for $t > m$, $t \wedge \tau =m \wedge \tau$.

We will change our reparameterization slightly by setting
\[          \p_t \sigma_t =  \tilde K_t , \;\;\;\; \tilde K_t = K
_{t \wedge \tau}. \]
 This will have no affect on times $t \leq \tau$ which is all that is
 important for us.  The advantage of this change is that we can write
 \[   c \leq \tilde K_t \leq C, \;\;\;\;\;
        |\tilde K_t - \tilde K_s| \leq u(t-s) , \]
  for a continuous function $u$ with $u(0+) = 0$.
   Since we are considering
  only $t  \leq \tau$, we will write $K_t$ rather than $\tilde K_t$.

 In order to prove \eqref{oct22.1}, we consider $\tilde L_t = L_{\sigma_t}$.
 This is a supermartingale, and as noted in Section \ref{nathalfsec}, we can write
 \begin{equation}  \label{sept25.3}
    \tilde \Theta_t = \lim_{n \rightarrow \infty} \tilde \Theta_{t,n},
   \end{equation}
  where
  \[  \tilde \Theta_{t.n} =
    \sum_{j < t2^n}  \tilde L\left(\frac{j-1}{2^n}, \frac{j}{2^n}
      \right)
 ,
 \;\;\;\;\tilde L(r,s) = \E\left[\tilde L_r - \tilde L_s
 \mid \tilde \F_r\right]. \]
 In this case, the Doob-Meyer theorem only
 gives a  weak-$L^1$ limit in  \eqref{sept25.3}, but this is
 all that we need.  We will show that
 \begin{equation}  \label{oct10.4}
 \E\left[|\hat \Theta_{t \wedge \tau,n} - \tilde \Theta_{t \wedge \tau,n}  |\right] \rightarrow 0.
 \end{equation}
 Since weak limits are unique, this implies that $\hat \Theta_{t \wedge \tau} =
 \tilde \Theta_{t \wedge \tau}$.
 To establish \eqref{oct10.4}, we claim that it
 suffices to find for fixed $t$ and $\tau$,  a uniformly
 bounded sequence of random variables $J_n$ with $J_n \rightarrow 0$
 with probability one and
 \[            |\hat \Theta_{t \wedge \tau,n} - \tilde \Theta_{t \wedge \tau,n}  |
   \leq J_n \, \hat \Theta_{t \wedge \tau,n} .\]
 Indeed, since
 $\hat \Theta_{t,n}
 \rightarrow \hat \Theta_t$ in $L^1$, the random variables
 $\{\hat \Theta_{t,n}: n=1,2,\ldots\}$ are uniformly integrable.
Using this and the fact
that the $J_n$ are uniformly bounded with $J_n \rightarrow 0$,
we see that
$\E[J_n \, \hat \Theta_{t,n}] \rightarrow 0$.
       If we let
     $u_n = u(2^{-n}),$
 $K_{j,n} = K_{(j-1)2^{-n}}$, then  if $\tau > (j-1)2^{-n}$,
\[  L\left(\sigma_{j-1,n},  \sigma_{j-1,n}+ \frac{K_{j,n} -  u_n}{2^n}
        \right)\leq
          \tilde L\left(\frac{j-1}{2^n}, \frac{j}{2^n}\right)
       \leq   L\left(\sigma_{j-1,n},  \sigma_{j-1,n}+ \frac{K_{j,n} +  u_n}{2^n}
        \right).\]
Using Lemma \ref{laterlemma}, we can conclude that
\[   \tilde L\left(\frac{j-1}{2^n}, \frac{j}{2^n}\right)
  = L\left(\sigma_{j-1,n},  \sigma_{j-1,n}+ \frac{K_{j,n}  }{2^n}
        \right)\, [1 + O(u_n)].\]
        Using \eqref{sept25.3}, we see that
\[    \tilde \Theta_t = \lim_{n \rightarrow \infty}
    \sum_{j < t2^n}  L\left(\sigma_{j-1,n},  \sigma_{j-1,n}+ \frac{K_{j,n}  }{2^n}
        \right)
 .\]
 We recall that
\begin{eqnarray*}
   L\left(\sigma_{j-1,n},  \sigma_{j-1,n}+ \frac{K_{j,n}  }{2^n}
        \right) & = & \E\left[\Theta_{t + r^2 2^{-n}
        } - \Theta_t \mid \F_t\right]\\
   & = &
        \int_\Half |\tilde f_t'(z)|^d \, G_t^{r^22^{-n}}(z) \, dA(z) ,
         \end{eqnarray*}
    with $t = \sigma_{j-1,n}$ and $r^2= K_{j,n}$.
    Using \eqref{oct2.6}, this can also be written as
 \[    r_{{j-1,n}}^{d } \int_\Half |\tilde f_{j-1,n}'(r_{j-1,n}
 z)|^d \, G_t^{2^{-n}}(z) \, dA(z), \;\;\;\;
  \tilde f_{j,n} = \tilde f_{j2^{-n}}, \;\;\; r_{j,n} = \hat F_{j2^{-n}}'(0) . \]
   Arguing as in Proposition \ref{prop.oct9}, we get
\begin{equation}  \label{oct9.5}
\tilde \Theta_t= \lim_{n \rightarrow \infty}
    \sum_{j < t2^n} r_{{j-1,n}}^{d } \int_{\Half_n} |\tilde f_{j-1,n}'(r_{j-1,n}
 z)|^d \, G_t^{2^{-n}}(z) \, dA(z)
 ,
 \end{equation}
 where as before $\Half_n = \{z \in \Half: |z| \leq n \, 2^{-n/2}\}.$

 By comparing \eqref{oct9.3} and \eqref{oct9.5}, we see that
 it suffices to prove the following.  For each $m$, there exists
 a sequence $u_n \downarrow 0$, such  that if $ t
  \leq \tau_m$ and $r=\hat F'_t(0)$, then for $z \in \Half_n$,
 \[\left| r   |\tilde f_t'(rz)| -
  \left| F'\left(\eta\left(t
      \right) \right) \right|\,  | \hat f_{t}'(z)| \right|
        \leq u_n \, r   |\tilde f_t'(rz)| .\]
        The sequence $u_n$ can depend on the path but the
        estimate must hold uniformly for all $t \leq \tau_m$
        and $z \in \Half_n$.
     We write this shorthand as
\begin{equation}  \label{oct10.2}
 r   |\tilde f_t'(rz)|
      =  \left| F'\left(\eta\left(t
      \right) \right) \right|\,   |\hat f_{t}'(z)|
      \, [1+o(1)] ,
      \end{equation}
   with the above uniformly implied.  We claim
   that
   \[   |F'(\hat f_t(z)) | =  \left| F'\left(\eta\left(t
      \right) \right) \right| \, [1 + o(1) ]. \]
Since $\eta(t) = \hat f_t(0)$, this
 follows from distortion estimates on $F$
   provided that we have uniform bounds
   on  $|\hat f_t(z) - \hat f_t(0)|.$
 But these are provided by Lemma \ref{uniformlemma}.
 Distortion estimates also imply
 \[          |\hat F_t'(z)| = r \, [1 + o(1)], \]
 \[          rz = \hat F_t(z) \, [1 + o(1) ], \]
 \[        |\tilde f_t'(rz)| = |\tilde f_t'(F_t(z))| \, [1 + o(1) ]. \]
 Therefore, \eqref{oct10.2} becomes
    \[  |\tilde f_t'(\hat F_t(z))| \, |\hat F_t'(z)|
      =  |F'(\hat f_t(z)) |\,    |\hat f_{t}'(z)|
      \, [1+o(1)] . \]
      But $\tilde f_t \circ \hat F_t = F \circ \hat f_t$, so the
   chain rule implies that $ |\tilde f_t'(\hat F_t(z))| \, |\hat F_t'(z)|
      =  |F'(\hat f_t(z)) |\,    |\hat f_{t}'(z)|  $.

\subsection{A particular case}

 If $\gamma$ is an $SLE_\kappa$ curve from $0$
 to $\infty$ in $\Half$ and $r > 0$, then
 \[  \gamma^{(r)}(t) = \gamma(t+r) , \]
 is an $SLE_\kappa$ curve from $\gamma(r)$ to
 $\infty$ in $H_r$.   We would like to say that the
 natural parametrization in $H_t$ is the same as
 that in $\Half$, and we make this precise here.

 Suppose $\gamma(t)$ has driving function $U_t = - B_t$.  We
 recall that we can define $\Theta_t$ by first defining it for
dyadic rational $t$ by the $L^1$ limit
\[         \Theta_t = \lim_{n \rightarrow \infty}
    \Theta_t^{(n)} , \]
  and for other $t$ it is defined by continuity.  Here we
  use the fact that the natural parametrization is continuous
  with respect to the capacity parametrization.  A key
  observation, is that given the Brownian motion $B_t$,
  there is a well defined natural parametrization
  $\Theta_t$ for all $t \geq 0$,
  up to a single null event.

  If $\tau$ is a stopping time,   let
 \[    \gamma^{\tau}(t) = \gamma(\tau + t), \;\;\;\;
 \eta^{\tau}(t) = g_\tau(\gamma^{\tau}(t) ) - U_\tau, \]
 and note that
 \[    \gamma^{\tau}(t) = f_\tau(\eta^\tau(t)). \]
 The strong Markov property implies that
  $\eta^{\tau}$ is an $SLE_\kappa$ curve from
 $0$ to $\infty$ with driving function $U^{\tau}_t
 = - B^{\tau}_t$ where $B^{\tau}$
is  the Brownian motion,
  \[         B_t^{\tau}  =  B_{t + \tau} - B_\tau . \]
Let $\Theta^{\tau}_t$ denote the corresponding version
of the natural parametrization defined as in the previous
paragraph.   Since $f_\tau: \Half \rightarrow H_\tau$
is a conformal transformation, we can view
$\gamma^{\tau}$ as an $SLE_\kappa$ curve from
$\gamma(\tau)$ to $\infty$.  The natural length of $\gamma^r[0,t]$
considered as a curve in $H_\tau$ is
\[   \int_0^t |f_\tau'(\eta^{\tau}(r))|^d \, d\Theta^{\tau}_r.\]

\begin{lemma} Suppose $r > 0$,   and $\tau$ is
a stopping time with $\tau \leq r$.  Let
\[        Z_n(s,t)= \sum L(\tau + j2^{-n},
 \tau + (j+1)2^{-n} ) , \]
 where the sum is over all $j$ with $s \leq j/2^{n} \leq t$.
 Then with probability one for all $0 < s < t$,
 \[   \lim _{n \rightarrow \infty}
    Z_n(s,t) = \Theta_{\tau +t} - \Theta_{\tau + s} .\]
\end{lemma}

\begin{proof} By continuity it suffices to prove
the result for fixed $s < t$ that are dyadic
rationals.
Apply the same proof as 
for Theorem \ref{maytheorem2} to the supermartingale
\[   \Psi^{\tau}_t = \Psi_{t+ \tau} -\Psi_\tau. \]
\end{proof}

\begin{proposition}  If $\tau$ is a stopping time
with $\Prob\{\tau < \infty\} = 1$, then
with probability one, for all $0<s<t$,
\begin{equation}  \label{dec15.1}
    \Theta_{t+\tau} - \Theta_{s+\tau} =
   \int_s^t |f_\tau'(\eta^{\tau}(r))|^d \, d\Theta^{\tau}_r.
   \end{equation}
 \end{proposition}

 \begin{proof}

  Without loss of generality, we may assume
 $\tau$ is a bounded stopping time (otherwise, apply the
 proposition to $\tau \wedge k$ and let $k \rightarrow
 \infty$).   We  note that
 for each $s$, the natural parametrization
 after time $s$ is supported on the curves in $H_s$.
 Since both sides  of \eqref{dec15.1}
 are continuous in $t$ with
 probability one, it suffices to show that for positive
 dyadic rationals $s < t$ such that $\gamma[s,t] \subset
 H_\tau$
 \[  \Theta_{t + \tau} - \Theta_{s + \tau}
   =  \int_s^t |f_\tau'(\eta^{\tau}(r))|^d \, d\Theta^{\tau}_r.\]
Let $t_{j,n} = \tau + j2^{-n}$.  By the lemma and \ref{prop.oct9},
\[  \Theta_{t+\tau} - \Theta_{s+\tau}
 = \lim_{n \rightarrow \infty} \sum_{s+\tau \leq t_{j,n}
  < t+\tau} \int_{\H_n}  |f_{t_{j,n}}'(z)|^d \, G^{2^{-n}}(z)
   \, dA(z) . \]
   
   In Proposition \ref{prop.oct9}, we have a $L^1$ limit, but
   we can assume it is an almost sure limit by taking a subsequence
   if necessary (for notational ease, we will assume it is an
   almost sure limit, but the remainder of this proof could be
   done along a subsequence if necessary).
    Take $\delta_n$ in the definition of $\H_n$ such 
that $\diam(\H_n) \rightarrow 0$ as $n \rightarrow \infty$. 

   Define $\tilde f_{j,n}$ by $f_{t_{j,n}} =
   f_\tau \circ \tilde f_{j,n}. $  Then by Proposition \ref{prop.oct9},
\begin{equation} \label{donut}
 \Theta_r^{\tau} =  \lim_{n \rightarrow \infty} \sum_{ t_{j,n}
    \leq r
  } \int_{\H_n}  |\tilde f_{j,n}'(z)|^d \, G^{2^{-n}}(z)
   \, dA(z) . 
\end{equation}

By continuity of natural length, we can write the right-hand side of \eqref{dec15.1} as

\[ \lim_{n \rightarrow \infty} \sum_{s+\tau \leq t_{j,n} \leq t+\tau} \int_{\H_n} |f'_\tau(\eta^\tau(t_{j,n}-\tau))|^d \, |\tilde f_{j,n}'(z)|^d \, G^{2^{-n}}(z) \,dA(z). \]
Hence, the difference of the two sides of \eqref{dec15.1} is
\begin{equation}     \label{additivedif}
\lim_{n \rightarrow \infty} \sum_{s+\tau \leq t_{j,n} \leq r+\tau} \int_{\H_n}  \left[|f'_\tau(\eta^\tau(t_{j,n}-\tau))|^d-|f'_\tau(\tilde f_{j,n}(z))|^d\right] \, |\tilde f_{j,n}'(z)|^d \, G^{2^{-n}}(z) \,dA(z). 
\end{equation}
By \eqref{donut}, we see that this is bounded above by
\[      [\Theta_{t}^\tau - \Theta_{s}^\tau] \, \limsup_{n \rightarrow
\infty} K_n(\gamma), \]
 where 
\[ K_n(\gamma)=\max_{j,z \in \H_n} \{||f'_\tau(\eta^\tau(t_{j,n}-\tau))|^d-|f'_\tau(\tilde f_{j,n}(z))|^d|\}.
\]
Hence it suffices to show that $K_n(\gamma) \rightarrow 0$
with probability one.   


Note that because it is an almost sure expression we can show it goes to zero for a given curve $\g$ so our constants may depend on it. We have $\tilde f_{j,n}(0)=\eta^\tau(t_{j,n}-\tau)$ so because $z \in \H_n$ and $\diam(\H_n) \rightarrow 0$ this is equivalent to the fact that $f'_\tau$ is continuous in a neighborhood of $\eta^\tau[s,t]$. Using the assumption $\g[s,t] \subset H_\tau$, we see that $\eta^\tau[s,t] \subset \{\Im(z)>\epsilon\}$ for some $\e>0$ which depends on $\gamma$. Continuity of $f'_\tau$ on $\Im(z)>\epsilon$ is a consequence of Cauchy integral formula and we are done.

\end{proof}

 \section{Bounds on the two-point Green's function}  \label{boundssec}

  \subsection{Lower bound for the time-dependent Green's function}  \label{uglysec}

  In this section we prove Lemma \ref{may14.lemma1}.
  This is a generalization of a result in \cite{LZ} where it is shown that
 there exists $c < \infty$ such that for all $z,w \in \Half$,
  \[     G(z) \, G(w)  \leq c \, \left[\hat G(z,w) + \hat G(w,z)\right].\]
  Examination of that  proof shows that the same argument
  establishes the existence of $c < \infty$ such that for $|z|,|w| \leq 1$,
  \begin{equation}  \label{sept1.1}
      G(z) \, G(w)  \leq c \, \left[\hat G^{c,c}(z,w) + \hat G^{c,c}(w,z)\right].
      \end{equation}
  We will assume \eqref{sept1.1}.

  We write $z = x_z + iy_z,
 w = x_w + i y_w$.   We assume
\begin{equation}  \label{sept8.3}
        y_z^2 \leq 2as,\;\;\; y_w^2\leq 2at,
   \end{equation}
   for otherwise
 the left-hand side of \eqref{twopointer} is  $0$.   All constants
 in this proof are independent of $z,w,s,t$, but may depend on $a$.
 We will first handle the case when $t \ll   y_z^2$ or $s \ll y_w^2$.

\begin{itemize}
\item {\bf Claim 1}: There exist $\epsilon >0, c< \infty$ such
 that
  \[  G^s(z) \, G^t(w)  \leq   c\, \hat G^{ct,cs}(w,z)  \;\;\; \mbox{ if }
   \;\;\; t \leq \epsilon \, y_z^2 , \]
   and
   \[  G^s(z) \, G^t(w) \leq c \, \hat G^{cs,ct}(z,w) \;\;\;
    \mbox{ if } \;\; s \leq \epsilon y_w^2. \]
  \end{itemize}

 By scaling, we may assume that $1 = |z| \leq |w|$.
 We will prove the first inequality which is the harder of the
 two; the second can be proved similarly. We will  assume that
  $t \leq  y_z^2/(200a)$; we will later choose $\epsilon < 1/(200a)$.
Using \eqref{sept8.3}  we see that
 $t \leq s/100$,  and $y_z \leq 1$ implies that $t \leq 1/(200a) \leq 1/50$.
Using \eqref{sept8.3} again, we see that  $y_w \leq 1/10$,
 and hence $|w|\asymp |x_w|  \geq  \sqrt{99}/10$.
 Let $T = T_w$ and $\gamma = \gamma[0,T]$.
By Corollary \ref{aug31.corollary}, there exists  $c_1 < \infty$ such that
\[      \Prob_w^*\{ \diam(\gamma) \leq c_1 \, |w| \mid T \leq 2t\}
   \geq \frac 12 . \]
   Let  \[    E = \{\diam(\gamma) \leq c_1 \, |w|, T \leq 2t\}, \]
   and note that $\Prob_w^*(E) \geq \Prob_w^*\{T \leq 2t\}/2$.
  On the event $E$, we can use the Loewner equation 
  \eqref{loew}
  and conformal
  invariance arguments to estimate $Z_r(z)$ and $g_r(z)$
  for $0 \leq r \leq T$.  We
  list the estimates here omitting the straightforward proofs.  First,
  there exists $c_2 < \infty$ such that
\[ |U_r| \leq c_2 \, |w|,\;\;\;\; 0 \leq r\leq T.\]
Also,
 \[  |Z_r(z)| \geq  Y_r(z) \geq \frac{y_z}{2}, \;\;\; 0 \leq r \leq T, \]
\[ |\p_r g_r(z)| \leq \frac{a}{|Z_r(z)|} \leq \frac{2a}{y_z}, \;\;\; 0 \leq r \leq T, \]
\[ |g_T(z) - z| \leq \frac{ 4at}{y_z } \leq \frac{y_z}{50}  \leq \frac 1{50}, \]
 \[ |Z_T(z)| \leq |g_T(z)| +|U_T(z)| \leq \frac{51}{50} + c_1|w|
    \leq  c_2 |w| , \]
\[ |\p_r \log g_r'(z)|  \leq \frac{a}{|Z_r(z)|^2} \leq \frac{ 4a}{y_z^2},\]
\[\left| \,\log |g_T'(z)| \, \right|\leq \frac{8at}{y_z^2} \leq \frac 1{25} .\]
If   $S_T(z)  = \sin \arg Z_T(z)$, then
   \[    S_T(z) = \frac{Y_T(z)}{|Z_T(z)|}  \geq \frac{c_3 \, y_z}{|w|}. \]
   \[  |g_T'(z)|^{2-d} \, G(Z_T(z)) = |g_T'(z)|^{2-d}
    \,  Y_T(z)^{1-4a} \, S_T(z)^{4a-1} \geq c_4 \, |w|^{1-4a} \,
      G(z). \]

   Using Lemma \ref{may19lemma} we can find $c_5,\beta_1$ such
   that
   \[        |g_T'(z)|^{2-d} \, G^s(Z_T(z))  \geq c_5 \, e^{-\beta_1|w|^2/s}
    \, G(z). \]
  Therefore, we get
  \[   G^{2s,2t}(z,w) \geq c_6 \, e^{-\beta_1|w|^2/s} \, G^s(z) \, G^{2t}(w) . \]
  Using the last inequality in Lemma \ref{may19lemma}, we can see that
  there exists $\delta$ such that
  \[ G^{2s,2t}(z,w) \geq c_7 \,  G^s(z) \, G^{\delta t}(w) . \]
If we let $\hat t = \delta t$, then we can rewrite this as
\[    G^{2s,(2/\delta) \hat t} \geq c_7 \, G^s(z) \, G^{\hat t}(w) , \]
which is now valid if $\hat t \leq \delta \, y_z^2/(200a)$. This establishes
the claim with $\epsilon = \delta/(200a)$.

\medskip

For the remainder of the proof we fix $0 < \epsilon_0 <1/2a$ such that Claim 1
holds, and we assume that
$t \geq \epsilon_0 \, y_z^2, s \geq \epsilon_0 y_w^2$, and hence
\begin{equation}  \label{sept8.5}
   s \geq \frac{y_z^2}{2a}, \;\;\;\; t \geq \frac{y_w^2}{2a}, \;\;\;\;
  s \wedge t \geq \epsilon_0 \, (y_z \vee y_w)^2 \end{equation}

 \begin{itemize}
\item   {\bf Claim 2}.
There exist $l,\rho>0$ such that the following holds.  Suppose $z = x+iy$
with $|z| \leq 1$ and $y^2/a \leq s $. Let $V$ denote
the event
\[ V = V_{z,l} = \left\{\gamma[0,T_z] \subset \{z' : |z'| \leq l \} \right\}.\]
Then
\begin{equation}  \label{sept1.3}
\Prob_z^*(V \mid T_z \leq s ) \geq \rho.
\end{equation}
\end{itemize}
This  was proved in Corollary
\ref{aug31.corollary}.
We fix $l \geq 8$ and $\rho> 0$ such that \eqref{sept1.3} holds.

\medskip

%
%

\begin{itemize}
\item
   {\bf Claim 3}.   There exist $c_1 , \beta_1 < \infty$ such that
  if $ |w| \geq 2l|z|$, then
 \[ G^s(z) \, G^t(w) \leq c_1 \, \hat G^{\beta_1 s, \beta_1 t}
   (z,w) . \]
   \end{itemize}

 By scaling we may assume $|z| = 1$.
  Let
\[   u = \min\left\{ 4s, \frac{2t}{a\epsilon_0}, 8 \right\},\]
and note that
$y_z^2 \leq au/2$.
      Let $T = T_z, \gamma = \gamma[0,T]$  and let $V$ be  as above.
Then
\[  \Prob_z^*(V \mid T \leq u) \geq \rho . \]
 On the event $V$,  since $T \leq 8$,
 \[   \gamma  \subset \{x+iy: 0 < y  \leq \sqrt {4a}, \;\;\;
   -l \leq x \leq l \}. \]
 Standard conformal mapping estimates imply that
 \[       |X_T(w)| \asymp  |x_w| , \;\;\;\; Y_T(w) \asymp y_w, \;\;\;\; |g_T'(w)|
 \asymp 1.\]
So if $u \in \{4s,8\}$ because we get $G^u(z)>cG^s(z)$ the argument proceeds in the same way (actually, somewhat more easily)  as Claim 1.\\

If $u=\frac{2t}{a\epsilon_0}$ then if we have $\beta_1$ then we can find $\beta_2$ such that

\[ cG^{\beta_1s,\beta_1t}(z,w) \geq cG^{\beta_2t,2\beta_2t}(z,w) \geq c G(z)\E^*_z[M_T^{\beta_2t}(w)\;|\; T \leq \beta_2t,V]\Prob[T \leq \beta_2t]
\]
Then again we have same estimates on $|X_T(w)| , \; Y_T(w)$ and $|g_T'(w)|$ as above, so by lemma \ref{may19lemma} the last term is bigger than $cG(z)G^{\delta \beta_2t}(w)e^{-\frac{5}{\beta_2t}}$ for some $\delta>0$ fixed. Again using Lemma \ref{may19lemma} we need to show existance of $\beta_2$ such that $\frac{-5|w|^2}{\delta\beta_2}-\frac{5}{\beta_2} \geq \frac{-|w|^2}{8}$ which we can do by $|w|>2l$ so we are done.

\medskip

\begin{itemize}
\item {\bf Claim 4.} There exists $c < \infty$ such that if $|z|,|w| \leq 4l$
and $t \geq (y_z \vee y_w)^2/(2a)$, then
\[            G(z) \, G(w)
 \leq c \, e^{\beta/t} \, G^{ct,ct}(z,w) .\]
\end{itemize}

It
suffices to prove the result for $t$ sufficiently
small for otherwise
we can use \eqref{sept1.1}.  For the moment we assume
 $t \leq 1/(100a)$.  Using Claim 3, we see
that  it suffices to prove the
estimate for  $z = 1 + iy_z, w = x_w + i y_w$ with $y_z,y_w \leq 1/10$
and $1 \leq |x_w| \leq 4l$.  In this case, $G(z) \asymp y_z^{4a-1},
G(w) \asymp y_w^{4a-1}$. Let
\[  \sigma = \inf\{t: \Re[\gamma(t)]= 1 -\sqrt t\},\]
and let $E$ be the event
\[   E = E_t = \left\{\sigma \leq t, \;\; \, \gamma_\sigma \subset  \{\Re[z']
\geq -1/2\} \right\} . \]
Arguing as in Lemma \ref{may19lemma}, we can see that
\begin{equation}  \label{clinton}
  \Prob\{\sigma \leq t, \Re[\gamma(s)] \geq -1/2 \mbox{ for }
0 \leq s \leq t \} \geq c_1 \, e^{-\beta_1/t}.
\end{equation}
Since $\hcap[\gamma_\sigma] = a\sigma$, we see that on the event $E$,
\[             \gamma_\sigma\subset V := \{x+iy:
    -1/2 \leq x \leq 1-\sqrt t, 0 \leq y \leq \sqrt{2at} \}. \]
By the strong Markov property, for every $s > 0$,
\[   G^{s+t,s+t}(z,w) \geq \E\left[|g_\sigma'(z)|^{2-d} \, |g_\sigma'(w)|^{2-d}
  \, G^{s,s}(Z_\sigma(z),Z_\sigma(w))\, 1_E \right].\]
We now list some deterministic estimates that hold on the event
$E$. The constants $0 < c_1 < c_2<\infty$ can be chosen
  uniformly over every curve
$\gamma$ with $\gamma_\sigma \subset V$ and $\Re[\gamma(\sigma)]
 = 1 - \sqrt t$.  We follow the statement of each estimate
 with a brief
 justification.  We write $g = g_\sigma$.

 \begin{itemize}

 \item $  c_1 y_z \leq  Y_\sigma(z) \leq y_z , \;\;\;c_1 y_w
   \leq  Y_\sigma(w) \leq  y_w . $
 The argument is the same for $z$ and $w$.  By conformal
 invariance and the fact that $g(z') \sim z'$ as $z'
 \rightarrow \infty$,
 \[    Y_\sigma(z) = \lim_{R \rightarrow \infty} R \, \Prob^{z}
 \{ \Im[B_{\tau}] = R \} , \]
 where $B$ is a complex Brownian motion and $\tau = \tau_{R,\gamma}$
is the first time $r$ that $B_r \in \R \cup \gamma_\sigma \cup
 \{ \Im(w') = R\}$.  The probability that a Brownian motion
starting at $z = 1 + iy_z$ reaches $\{\Im(w') = 2\sqrt{at}\}$ before
time $\tau$ is bounded below by $c y_z/\sqrt t$ and given this the probability
that $\Im(B_{\tau}) = R$ is greater than $c \sqrt t/R$.

\item      $c_1 \leq  |g'(z)|,  |g'(w)|  \leq c_2 .$

It follows from the previous estimate that $c_1 \leq g'(1),
g'(x_w) \leq 1$.  We can then use the Schwarz reflection and distortion theorem.

\item $        c_1 \sqrt t \leq |Z_\sigma(z)| \leq c_2 \sqrt t .$

For the lower bound consider $g$ (extended by
Schwarz reflection) on the disk of radius $\sqrt t$ about
$z$.  The image of this disk contains a disk of radius
$|g'(z)| \sqrt t/4$.  Therefore, $|Z_\sigma(z)| \geq
|g'(z)| \sqrt t/4 \geq c \, \sqrt t.$ For the upper
bound consider $\zeta = 1 + i\sqrt t$.  Using conformal
invariance, it is easy to see that $S_\sigma(\zeta) \geq
c$ and hence $|Z_\sigma(\zeta)| \asymp Y_\sigma(\zeta)
\leq \sqrt t$.  Since $|g'|$ is uniformly bounded on the
line segment $[1,1+i \sqrt t]$, we get the upper bound.

\item  $G(Z_\sigma(z)) \geq c_1 \, G(z), \;\;\;
   G(Z_\sigma(w)) \geq c_1 \, G(w) .$

This is immediate from the estimates we have already established.
 \end{itemize}

Since $|g_\sigma'(z)| \asymp |g_\sigma'(w)| \asymp 1$,
\eqref{clinton} implies for all $\beta > 0$,
\begin{eqnarray*}
   G^{(\beta+1)t,(\beta+1)t}(z,w) &  \geq & c\, \E\left[
   G^{\beta t,\beta t}(Z_\sigma(z),Z_\sigma(w))\, 1_V \right]\\
   & \geq & c \, e^{-\beta_1/t} \,
      \E\left[  G^{\beta t,\beta t}(Z_\sigma(z),Z_\sigma(w)) \mid
      V \right].
      \end{eqnarray*}
 We now consider two cases.  First, if $|Z_\sigma(w)| \geq
 2l |Z_\sigma(z)|$, then Claim 3 gives the existence of $c, \beta$
 such that
 \[           G(Z_t(z)) \, G(Z_t(w)) \leq  c\,
 G^{\beta t,\beta t}(Z_t(z),Z_t(w)) . \]
 If $|Z_\sigma(w)| \leq 2l |Z_\sigma(z)|$, then 
 distortion estimates imply that $|Z_\sigma(w)|
 \asymp \sqrt t$.  Scaling gives
 \[     G^{\beta t,\beta t}(Z_t(z),Z_t(w))
   = t^{ 2-d} \,
   G^{\beta,\beta}(Z_t(z)/\sqrt t, Z_t(w)/\sqrt t).\]
 If we choose $\beta$ sufficiently large,
 \begin{eqnarray*}
    G^{\beta,\beta}(Z_t(z)/\sqrt t, Z_t(w)/\sqrt t)
    & \geq & c\,  G(Z_t(z)/\sqrt t) \, G( Z_t(w)/\sqrt t)\\
       & \geq & c t^{d-2} \, G(z)\, G(w).
       \end{eqnarray*}
Hence, we can find $c, \beta$ such that
\[  G^{\beta t, \beta t}(z,w)
\geq c \, e^{-\beta_1/t} \,
       G(z) \, G(w) . \]

       \medskip

Claim 5 finishes the proof of the lemma.

 \begin{itemize}
 \item {\bf Claim 5.}  There exists $c < \infty$ such that if $|z|,|w| \leq 4l$,
 and $s,t$ satisfy \eqref{sept8.5}, then
 \[  G^s(z) \, G^t(w) \leq c \, G^{cs,ct}(z,w).\]
 \end{itemize}

 By symmetry we may assume $t \leq s$.   By \eqref{sept8.5}, we can find
 $r$ such that $rt \geq (y_z \vee y_w)^2/(2a)$.  From Claim 4 we can find
 (different) $c,\beta$ such that
 \[            G(z) \, G(w)
  \leq c \, e^{\beta/t} \, G^{ct,ct}(z,w)  . \]
  We also know that there exist $c,\beta_2$ such that $G^t(w) \leq
 c\, e^{-\beta_2/t} \, G(w).$  By choosing $c$ even larger in the
 last displayed formula, we can guarantee that
 \[            G(z)\, G(w) \leq c e^{\beta_2/t} \, G^{ct,ct}(z,w) . \]
 Therefore,
 \[  G^s(z) \, G^t(w) \leq G(z) \, G^t(w) \leq
   c \, e^{-\beta_2/t} \, G(z) \, G(w) \leq c \, G^{ct,ct}(z,w)
    \leq c \, G^{cs,ct}(z,w).\]

\subsection{Proof of Theorem \ref{sept8.theorem} }  \label{uglysec2}

In \cite{LZ} it is shown that there exists $c < \infty$ such that for all
$z,w\in \Half$,
\[      G(z) \, G(w) \leq c \, G(z,w) . \]
In \cite{LW}, a bound in the other direction was given: if $|w| = 1$
and $|z-w| \leq 1/2$, then
\[       G(z,w) \leq c \, |z-w|^{d-2} . \]
However, these papers did not give precise estimates in the case
where $S(z),S(w)$ are small.  Our proof will use the ideas
in \cite{LW}.
Throughout this
   section, $\gamma$ will denote an $SLE_\kappa$ curve and
   \[  \gamma_t = \gamma(0,t], \;\;\;\Delta_t(z) = \dist(z,\gamma_t),
  \;\;\; \Delta(z) = \Delta_\infty(z). \]
 We write
$\asymp$ to indicate that quantities are bounded by constants
where the constants depend only on $\kappa$.
     We recall
that in \cite{LW} it is shown that for each $z,w$, there
exist $\epsilon_z,\delta_w$ such that if $\epsilon < \epsilon_z,
\delta < \delta_w$,
\begin{equation}  \label{nov17.1}
   \Prob\{ \Delta(z)\leq \epsilon\} \asymp
     G(z) \, \epsilon^{2-d} , \;\;\;\;
       \Prob\{ \Delta(w) \leq \delta\} \asymp G(w)
        \, \delta^{2-d} ,
        \end{equation}
        \begin{equation}  \label{nov17.2}
  \Prob\{ \Delta(z)\leq \epsilon,
     \Delta(w)\leq \delta\} \asymp G(z,w)
     \, \epsilon^{2-d} \, \delta^{2-d}.
     \end{equation}

    When estimating
$\Prob\{\Delta(z) \leq \epsilon\}$ there are two regimes.
The interior or bulk regime, where $\epsilon \leq \Im(z)$ can
be estimated using Proposition \ref{mar14.prop1} since in this
case $ \Delta(z)   \asymp
 \Upsilon(z)  $.  However for the boundary
regime $\epsilon > \Im(z)$,
one needs a different result.

\begin{lemma}  \label{oct23}
There exists $0 < c_1 < c_2 < \infty$ such that if $ 0 < y \leq 1/4$ and
$\sigma = \inf\{t: |\gamma(t) - 1| \leq 2y\}$, then
\[  c_1 y^{4a-1} \leq \Prob\{\sigma < \infty, S_\sigma(1 + iy)
\geq 1/10\} \leq
\Prob\{\sigma < \infty\} \leq c_2 y^{4a-1}.\]
\end{lemma}

\begin{proof}  The bound $\Prob\{\sigma < \infty\} \asymp
y^{4a-1} $ can be found in a number of places.  A proof which includes a
  proof of the first inequality can be found in \cite{LZ}.  The first
  inequality is Lemma 2.10 of that paper.
\end{proof}

   We will prove Theorem \ref{sept8.theorem} in a sequence
   of propositions.   We assume $|z| \leq |w|$ and let
   \[ q = |w-z|, \;\;\;\;
   \beta = (4a-1) - (2-d) = 4a + \frac{1}{4a} -2> 0. \]
   It will be useful to define a quantity that allows us
to consider  the boundary and interior cases simultaneously.
 Let
\[  \newernot_t(z)   = \Delta_t(z)^{4a-1}  \;\; \mbox{ if } \;\;
  \Delta_t(z)\geq \Im(z), \]
  \[ \newernot_t(z) =   \Im(z)^{4a-1} \,  \left[\frac{\Delta_t(z)}{\Im(z)}\right]^{2-d}
  \;\;\mbox{ if } \;\; \Delta_t(z) \leq \Im(z) , \]
     and let $\newernot(z) = \newernot_\infty(z)$.
    Note that $\newernot_0(z) = |z|^{4a-1} $, and scaling implies that
    the distribution of $\newernot(rz)$ is the same as that
    of $r^{4a-1}\newernot(z)$.  Since $4a-1 > 2-d$,
    we see that
    \[  \Delta_t(z)^{4a-1}
    \leq  \Phi_t(z)   . \]
    The next lemma combines the interior and boundary estimates into one
    estimate.

    \begin{lemma}  There exist $0 < c_1 < c_2 < \infty$ such
    that for all $z \in \Half$ and $0 < \epsilon \leq 1$,
 \begin{equation}  \label{sept5.1}
  c_1 \epsilon \leq \Prob\{\newernot(z) \leq \epsilon \, \newernot_0(z)\}
\leq c_2  \epsilon .
\end{equation}
   \end{lemma}

   \begin{proof}  Let $z=x+iy$.
    By scaling we may assume that $|z| = 1$ and hence $\newernot_0(z) = 1,
   S(z) = y$.  Let $\Delta = \Delta_\infty(z), \newernot = \newernot_\infty(z)$.
  Proposition \ref{mar14.prop1}
    and Lemma \ref{oct23}
    imply that
   \[  \Prob\{\Delta  \leq \epsilon \} \asymp \epsilon^{4a-1} ,\;\;\;\;
      \epsilon \geq y, \]
    \[   \Prob\{\Delta  \leq \epsilon  \} \asymp y^{4a-1} \, [\epsilon/y]^{2-d}
      ,\;\;\;\;
      \epsilon \leq y. \]
   If $\epsilon \geq y$, then
\[
   \Prob\{\newernot  \leq \epsilon^{4a-1}   \}
    = \Prob\{\Delta  \leq \epsilon  \} \asymp \epsilon^{4a-1} .\]
   If $\epsilon \leq y $, then if $u = (4a-1)/(2-d)$,
\[
     \Prob\{\newernot \leq \epsilon^{4a-1}   \}
     =   \Prob\{ y\, (\Delta /y)^{\frac{2-d}{4a-1}} \leq \epsilon   \}
    =  \Prob\{\Delta \leq y \, (\epsilon/y)^{u} \}
   \asymp y^{4a-1} \, [(\epsilon/y)^{u}]^{2-d}  = \epsilon^{4a-1}.
\]

    \end{proof}

The next proposition establishes
     an upper bound
   for the probability that an $SLE$ path gets close to a point and
   subsequently returns to a given crosscut.  It is a generalization
   of Lemmas 4.10 and 4.11 of \cite{LW}, and we use ideas
   from those proofs. Suppose $\eta:(0,1) \rightarrow \Half$
is a simple curve with $\eta(0+) = 0, \eta(1-) > 0$ and
write $\eta = \eta(0,1)$.   Let
$V_1,V_2$ denote respectively the bounded and unbounded
components of $\Half \setminus \eta$ and assume that $z
= x_z + i y_z
\in V_1, w  = x_w + i y_w\in V_2$.  Given the $SLE$ curve $\gamma$, one can show
(see \cite[Appendix A]{LW}) that there is a collection of open subarcs
$\{I_t: t < T_z \wedge T_w\}$ of $\eta$ with the following properties.
 Recall that $H_t$ is the unbounded component of $\Half \setminus \gamma_t$.
\begin{itemize}
\item  $I_0 = \eta$.
\item  $I_t \subset H_t$.  Moreover, $H_t \setminus I_t$ has
two connected components, one containing $z  $ and the other containing $w
 $.
\item  If $s < t$, then $I_t \subset I_s$.  Moreover, if
  $\gamma(s,t] \cap I_s = \emptyset$, then $I_t = I_s$.
\end{itemize}
If $\zeta \in \{z,w\}$, define
  stopping times $\sigma_{k} ,\sigma,\tau$  depending
on  $\zeta$  by
\[  \sigma_{k} = \inf\{t: \newernot_t(\zeta) = 2^{-k} \, \newernot_0(\zeta)\}, \;\;\;\;
  \sigma = \sigma_1, \]
  \[   \tau = \inf\{t \geq \sigma: \gamma(t) \in \overline {I_\sigma}
    \}. \]
  If $\tau < \infty$, let
  \[ J  = \frac{
 \newernot_\tau(\zeta)}{\newernot_0(\zeta)} . \]

\begin{proposition}  \label{brentprop}
There exists $c < \infty $ such that under the setup
above if $0 < \epsilon \leq 1/2$ and $\alpha =
2a - \frac 12 > 0$,
\[           \Prob\{\tau < \infty, J \leq \epsilon \}
   \leq c \,   \epsilon, \;\; \mbox{ if } \zeta = z, \]
\[     \Prob\{\tau < \infty , J \leq \epsilon \} \leq
    c \,  \epsilon \, \left[\frac{\diam(\eta)}{|w|}\right]^{\alpha} , \;\;
    \mbox{if } \zeta = w.\]
  \end{proposition}

\begin{proof}  The first inequality follows immediately from
\eqref{sept5.1}, as does the second if $|w| \leq 4 \,
\diam(\eta)$.   Therefore, using scaling, we may
 assume that $\diam(\eta) = 1, |w| \geq 4, \zeta = w$.  Let $C$
 denote the half-circle of radius $\sqrt{|w|}$ in $\Half$ centered at
 the origin.  Let $k_0$ be the largest
 integer such that $2^{-k_0} \geq S(w) = \Im(w)/|w|$.
 Let $\rho$ be the first time $t$  that $w$ is not in
 the unbounded component of $H_t \setminus C$.  
 Note that if  $\rho <
 T_w$,
 then $\gamma(\rho) \in C$.
Let
 \[     \hat J = \frac{\newernot_\rho(w)}{\newernot_0(w)}. \]
 Then, if $k$ is a positive integer and $\hat \sigma = \sigma_{k}$,
 \[   \Prob\{\tau < \infty, J \leq 2^{-k}\}
  \leq \Prob\{\hat  \sigma  < \rho  \wedge \tau, \tau < \infty\} + \sum_{j=1}^k
    \Prob\{\rho <  \hat \sigma  < \infty, 2^{-j} < \hat J \leq 2^{-j+1}\}. \]

  We will now show that
 \begin{equation}  \label{sept5.2}
 \ \Prob\{\hat  \sigma  < \rho  \wedge \tau, \tau < \infty\}
  \leq c \,2^{-k} \,  |w|^{-\alpha} .
  \end{equation}
Let $  H = H_{\hat \sigma}, I = I_{\hat \sigma}, g = g_{\hat \sigma},
U = U_{\hat \sigma}$.
 By \eqref{sept5.1},
 \[  \Prob\{\hat  \sigma  < \rho  \wedge \tau\} \leq
   \Prob\{\hat \sigma < \infty \} \leq c \, 2^{-k} . \]
   Let $H^*$ be the component of $H \setminus
 C$ containing $w$. On the event $\hat  \sigma  < \rho $,  $H^*$ is unbounded.
 Using simple connectedness of $H$, we can see that there is a
 subarc $l \in \p H^* \cap C$ that is a crosscut of $H$ and that separates
 $w$ from $I$ in $H$.  Since $l$ does not separate $w$ from $\infty$, $g(l)$
 is a crosscut of $\Half$ that does not separate $U$ from $\infty$; for ease
 let us assume that its endpoints are on $(-\infty,U]$.    Since $l$
 separates $w$ from $I$,  $l$ also separates $I$ from infinity in $H$.
 Therefore $g(l)$ separates $g(I)$ from $U$ and infinity in $\Half$.
 We use excursion measure (see \cite[4.1]{LW} for definitions and similar
 estimates) to estimate the probability that $\gamma[\hat \sigma, \infty)$
 returns to $I$.  The
 excursion measure between $g(I)$ and $[U,\infty)$ in $\Half \setminus
 g(I)$ is bounded above by the excursion measure between $g(I)$ and $g(l)$ in
 $\Half \setminus (g(I) \cup g(l))$ which by conformal invariance equals
 the excursion measure between $I$ and $l$ in $H \setminus(I \cup l)$.
 This in turn is bounded above by the excursion measure between
 $C$ and $\p \Disk$ in $\{\zeta \in \Half: 1 < |\zeta| < \sqrt{|w|}\}$ which
 is $O(1/\sqrt{ |w|}).$  Given this, we can use the estimate from Lemma
 \ref{oct23} to see  that the probability that an $SLE_\kappa$
 path from $U$ to $\infty$ in $\Half$ hits $g(I)$ is $O(|w|^{-(4a-1)/2})$.  Using
 conformal invariance, we conclude that
 \[     \Prob\{\tau < \infty \mid\hat  \sigma  < \rho  \wedge \tau\}
   \leq c \, |w|^{(1-4a)/2} \]
  which gives \eqref{sept5.2}.

We noted above that if $j \leq k_0$, then
 \[  \Prob\{\rho < \hat \sigma  < \infty\} = 0 . \]
 We will now show that if $j > k_0$,
 \begin{equation}  \label{sept5.3}
  \Prob\{\rho <  \hat \sigma  < \infty, 2^{-j} < \hat J \leq 2^{-j+1}\}
   \leq c \,2^{-k}  \, 2^{-j\alpha} \ |w|^{-\alpha}  .  \end{equation}
   The proposition then follows by summing over $j$.
Consider the event
 \[  E_j = \{\rho <   \infty, 2^{-j} < \hat J \leq 2^{-j+1}\}. \]
  Using
  \eqref{sept5.1}, we see that
\begin{equation}  \label{sept5.4}
     \Prob(E_j) \leq c \, 2^{-j}
     \end{equation}
 Let $H = H_\rho$.  On the event $E_j$, there is a subarc $l$ of
 $H \cap C$ that is a crosscut of $H$ with one endpoint equal
 to $\gamma(\rho)$ such that $l$ disconnects
 $w$ from infinity
in $H$.  Using this and the relationship between $S$ and
harmonic measure we see that $S_\rho(w)$ is bounded above by
the probability that a Brownian motion starting at $w$ reaches
$C$ without leaving $H$.  Using the Beurling estimate, we see
that the probability that it reaches distance $|w|/2$ from $w$  without leaving
$H$ is $O(2^{-j/2})$.  Given this, the probability that is reaches
$C$ without leaving $\Half$ is bounded above by $O(1/|\sqrt{|w|})$.
Therefore,  on the event $E_j$,
\[               S_\rho(w) \leq c \, 2^{-j/2} \, |w|^{-1/2}. \]
Using the strong Markov property and \eqref{sept5.1}, we see that
\[       \Prob\{\hat \sigma < \infty \mid E_j\} \leq c \, 2^{-j \alpha}
  \, |w|^{-\alpha} \, 2^{-(k-j)}, \]
  which combined with \eqref{sept5.4} gives \eqref{sept5.3}.

 \end{proof}

   In this section, we will consider
  two-point correlations.  The next proposition, which is the hardest to prove,
  shows that if $z,w$ are separated, then the events are independent
  up to a multiplicative constant.

  \begin{proposition}  \label{nov17.prop1}
   There exists $c < \infty$ such that if $|z| \leq 4|w|$,
  and $0 < \epsilon_z,\epsilon_w \leq 1$, then
  \[ \Prob\{\newernot(z) \leq \epsilon_z \,\newernot_0(z) ,  \newernot(w) \leq
  \epsilon_w\, \newernot_0(w)\}  \leq c\,
  \epsilon_z \, \epsilon_w, \]
  \end{proposition}

  \begin{proof}  By scaling, we may assume that $|z| \leq 1/2, |w| =  2$.   As
  before we consider a decreasing collection of arcs
  $\{I_t: t < T_z \wedge T_w\}$ with the following
  properties.
  \begin{itemize}
  \item  $I_0 =    \{\zeta \in \Half: |\zeta| = 1\}.$
  \item  For each $t$, $I_t$ is a crosscut of $H_t$ that separates $z$ from
  $w$ in $H_t$.
  \item  If $t > s$, then $I_t \subset I_s$.  Moreover, if $\gamma(s,t] \cap
  \overline{I_s} = \emptyset$, then $I_t = I_s$.
  \end{itemize}
  We define a sequence of stopping times as follows.
  \[      \sigma_0 = 0 , \]
  \[       \tau_0 = \inf\{t: |\gamma(t)| = 1\} = \inf\{t: \gamma(t)
  \in \overline{I_{\sigma_0}}\}. \]
  Recursively, if $\tau_k < \infty$,
  \[     \sigma_{k+1} = \inf\left\{t > \tau_k: \newernot_t(w)  = \frac 12 \, \newernot_{\tau_k}(w)
   \mbox{ or } \newernot_t(z)  = \frac 12 \, \newernot_{\tau_k}(z) \right\},\]
   and if $\sigma_{k+1} < \infty$,
   \[    \tau_{k+1} = \inf\{t \geq \sigma_{k+1} : \gamma(t) \in \overline{I_{\sigma_{k+1}}}\}.\]
  If one of the stopping times takes on the value infinity, then all the subsequent ones
  are set equal to infinity.  If $\sigma_{k+1} < \infty$, we set $R_k = z$ if
  $\newernot_{\sigma_{k+1}}(z) =\newernot_{\tau_k}(z)/2.$  Note that in this case,
  \[ \Delta_{\sigma_{k+1}} (z) \leq q\, \Delta_{\tau_k}(z) \mbox{ where }
    q =  2^{-\frac{1}{4a-1}}  < 1, \]
    and
  $\newernot_{t}(w) > \newernot_{\tau_k}(w)/2$ for all $t \leq \tau_{k+1}$.
  Likewise, we set $R_k = w$ if $\newernot_{\sigma_{k+1}}(w) =\newernot_{\tau_k}(w)/2.$

 It follows immediately from \eqref{sept5.1}, that for $r \leq 1/2$,
 \[
  \Prob\left \{    \newernot_{\tau_{0}}(z)
     \leq r \, \newernot_{0}(z) \right\} \leq c \, r , \]
   and for $r$ sufficiently small
   \[  \Prob\left \{  \newernot_{\tau_{0}}(w)
     \leq r \, \newernot_{0}(w) \right\} =0.\]

The key estimate, which we now establish,  is the following.
\begin{itemize}
\item  There exists $c,\alpha$ such that if $\tau_k < \infty, 0 <r \leq 1/2$ and $\zeta = x+iy \in \{z,w\}$,
then
\begin{equation}  \label{sept8.2}
  \Prob\left \{ \tau_{k+1} < \infty,  R_{k} = \zeta,  \newernot_{\tau_{k+1}}(\zeta)
     \leq r \, \newernot_{\tau_k}(\zeta) \mid \gamma_{\tau_k}\right\} \leq c \, r \, \newernot_{\tau_k}(\zeta)^\alpha.
     \end{equation}
  \end{itemize}
Let $H = H_{\tau_k}, I = I_{\tau_k},\hat  g = g_{\tau_k} - U_{\tau_k},
 \hat I = \hat g(I), \hat \zeta = \hat g(\zeta), \Delta = \Delta_{\tau_k}(\zeta),
 \newernot = \newernot_{\tau_k}(\zeta), \lambda = |g'(\zeta)|$.
 Recall that $\Delta^{4a-1} \leq \newernot$.
 If $ \newernot_t(\zeta) = r \newernot$ then $|\zeta - \gamma(t)| =
 \theta \, \Delta$ where
 \[   \theta = \left[\frac{y\wedge \Delta}{\Delta} \vee r\right]^{\frac 1{4a-1}}
   \,   \left [ \frac {r\Delta} { y \wedge \Delta} \wedge 1\right]^{\frac 1{2-d}}. \]
 Note that if $r \leq 1/2$ then $\theta \leq2^{-\frac 1{4a-1}}  < 1$.

 Let $V$ denote the closed
  disk of radius $2^{-\frac 1{4a-1}}  \Delta$ about $\zeta$, $y_*
 = y \vee (\theta\Delta/2)$ and
 $\zeta_* = x + y_*i \in V$.
  Note that $g$ is a conformal transformation defined on the open disk of radius
 $\Delta$ about $\zeta$ (if $y < \Delta$, then we extend $g$ by Schwarz
 reflection).  Hence by the distortion theorem, there exist $0 < c_1 < c_2 < \infty$
 such that if $\zeta_1 \in V$,
 \[                     c_1 \, \lambda \leq |\hat g'(\zeta_1)| \leq c_2 \, \lambda, \]
 \[                    c_1 \, \lambda \, |\zeta_1 - \zeta| \leq
             |\hat g(\zeta_1) - \hat \zeta| \leq c_2 \, \lambda \, |\zeta_1 - \zeta|. \]
%
%
  In particular, 
 \[               c_1 \, \lambda \, y \leq \Im \hat \zeta \leq c_2 \, \lambda \, y . \]

%
%
%
%
%
%
%
%
Note that   $\hat I$ is a crosscut of
 $\Half$ with one endpoint equal to zero.  We consider separately the cases where
 $\hat \zeta$ is in the bounded or unbounded component of $\Half \setminus \hat I$.

Let $E_1$ denote the event that
 $\hat \zeta$ is in the bounded component.  We claim that there
 exists $c < \infty$, such that for all $\hat \zeta' = \hat g(\zeta')
  \in \hat g( V)$,
\begin{equation}  \label{sept8.1}
        S(
       \hat \zeta') =    \frac{\Im(\hat \zeta') }{|\hat \zeta'|}
        \leq c \, \Delta^{1/2}
           \end{equation}
  To see this, assume for ease that $\Re[\hat \zeta'] \geq 0$
  and let $\Theta = \arg \hat \zeta'$.  Then     $\Im(\hat \zeta')/|\hat \zeta'| =
  \sin \Theta  \leq \Theta$ and $\Theta/\pi$ is the probability
  that a Brownian motion starting at $\hat \zeta'$ hits $(-\infty,0]$
  before leaving $\Half$.  This is bounded above  by
    the probability that a Brownian motion starting at $\hat \zeta'$
 hits $\hat I$ before leaving $\Half$.  By conformal invariance, this
 last  probability is the
 same as the probability that a Brownian motion starting at $ \zeta'$ hits
 $I$ before leaving $H$.   The Beurling
 estimate implies that this is bounded above by
 $c \Delta^{1/2}$.  This gives \eqref{sept8.1}.   Therefore, there exists
 $c$ such that if $|\zeta - \gamma(t)| = \theta \Delta$, then
 \[             \newernot(\hat g(\gamma(t))) \leq
       c \, \Delta^{(4a-1)/2} \, r \,|\gamma(t)| \leq c\, \sqrt{\newernot}
        \, r |\gamma(t)|. \]
       Using \eqref{sept5.1}, we see
  that
  \[  \Prob\{\newernot(\zeta) \leq r \, \newernot_{\tau_k}(\zeta), E_1 \mid
  \gamma_{\tau_k} \} \leq   c \, \sqrt{\newernot}\, r  . \]

%
%
%
%

  We now suppose that $\hat \zeta$ is in the unbounded component.
  By the same argument, for every $\hat \zeta' :=
  \hat g(\zeta') \in  \hat g( V)$, the probability
  that a Brownian motion starting at $\hat \zeta':= \hat g(\zeta')$ 
  hits $\hat I$ before leaving
  $\Half$ is bounded above by   $c\Delta^{1/2}$.  We will split
  into two subcases.  We first assume that
  \[            \Im(\hat \zeta') \leq  \Delta^{1/4} \, |\hat \zeta'|, \;\;\;\;
   \zeta' \in V. \]
  In this case, we an argue as in the previous paragraph to see that
  the probability $SLE_\kappa$ in $\Half$ hits $\hat g( V)$ is bounded above
  by  $c \, \newernot^{1/4} \, r$.  For the other case we assume that
  $\Im( \hat \zeta') \geq \Delta^{1/4} \,  |\hat \zeta'|$ for some $
  \hat\zeta' \in
  \hat g( V)$.  Using the Poisson kernel in $\Half$, we can see that the
  probability that a Brownian motion starting at $\hat \zeta'$ hits $\hat I$ before
  leaving $\Half$ is bounded below by a constant times
  \[                         \frac{\diam(\hat I)}{\Delta^{1/4} \, |\hat \zeta '|}. \] From
  this we conclude that
  \[              \diam(\hat I) \leq  c \, \Delta^{1/4}
   \, |\hat \zeta'|   . \]
  We appeal    to Proposition \ref{brentprop} to say that the probability
  that $SLE_\kappa$ in $\Half$ hits $\hat g(V)$ and then returns
  to $\hat I$ is bounded above by a constant times
  \[       r \, [\diam \hat I/|\hat \zeta'|]^{(4a-1)/2} \leq c\,  r \, \newernot^{1/8}.\]

  Given \eqref{sept8.2}, the remainder of the proof proceeds in the
  same way  as \cite[Section 4.4]{LW} so we omit this.

   \end{proof}

%
%
 \begin{proposition} \label{oct23.prop1}
There exist
 $0 < c_1 < c_2 < \infty$ such that if $|z| \leq |w|/4$,
 \[     c_1\,  G(z) \, G(w) \leq G(z,w) \leq c_2 \, G(z) \, G(w) .\]
 \end{proposition}

 \begin{proof}
 The bound $ G(z,w) \geq c \, G(z) \, G(w)$ was proved in \cite{LZ} so
 we need only show the other inequality. Proposition
 \ref{nov17.prop1} implies that for $\epsilon$ sufficiently small
 \[        \Prob\{\Delta(z) \leq \epsilon, \Delta(w) \leq \epsilon\}
   \leq c \, \Prob\{\Delta(z) \leq \epsilon\} \, \Prob\{\Delta(w) \leq \epsilon\} . \]
   Hence \eqref{nov17.1} and \eqref{nov17.2} imply that
   $G(z,w) \leq c \, G(z) \, G(w)$.
   \end{proof}

%

  The next estimate will be important even though it is not a very
  sharp bound for large $|z|,|w|$.

   \begin{proposition}  \label{sept9.prop1}
   For every $\epsilon >0$, there exists $c < \infty$ such that
   if $|z|,|w| \geq \epsilon$ and $|z-w| \geq \epsilon$, then
   \[         G(z,w) \leq c \, \Im(z)^{4a-1} \, \Im(w)^{4a-1} .\]
   \end{proposition}

   \begin{proof}  By scaling it suffices to prove the result
   when $\epsilon = 1$.  This can be done as the  proof
   of the previous proposition, so we omit
   the details.  The key step is to choose an appropriate splitting
   curve $I_0$.   We can choose $I_0$ either to be a half-circle with
   endpoints on $\R$ or a vertical line.  We choose $I_0$ so that
   $I_0$ separates $z$ and $w$ and $\dist(z,I_0), \dist(w,I_0)
   \geq 1/4$.
   \end{proof}

 We will now prove  Theorem
   \ref{sept8.theorem}.
    By scaling, we may assume that $|w| = 1$ and hence $q = |w-z|$.
If $q \geq 1/10$,  the conclusion is
 \[            G(z,w) \asymp G(z) \, G(w).  \]
 The bound $G(z,w) \geq c \, G(z) \, G(w)$ was done in \cite{LZ}.
The other inequality
   can be deduced from Propositions \ref{oct23.prop1}
 and \ref{sept9.prop1}, respectively, for $|z| \leq 1/4$ and $|z|
 \geq 1/4$.  Here we use the fact that $G(z) \geq \Im(z)^{4a-1}$
 for $|z| \leq 1$.

 For the remainder of the proof we assume $q \leq 1/10$, and hence
 $9/10 \leq |z| \leq 1$.
 Let $z = x_z + i y_z, w = x_w + i y_w,$  and
  $\zeta = x_w + i(y_w \vee q)$. Note that 
   $G(w) \asymp
  y_w^{4a-1}, G(z) \asymp y_z^{4a-1}$.
  Let $\sigma = \inf\{t:
 |\gamma(t) - w|  = 2 q \},$ and on the event $\{\sigma < \infty\}$,
 let $h = \lambda [g_\sigma- U_\sigma]
 $ where the constant $\lambda$ is chosen so that
 $\Im[h(\zeta)] = 1$. We write
 \[ h(\zeta) =\hat \zeta = \hat x_\zeta + i, \;\;\;\;
     h(z) = \hat z =  \hat x_z + i \hat y_z, \;\;\;\; h(w) =
     \hat w =  \hat x_w + i \hat y_w.\]
      Then
\begin{eqnarray*}
 G(z,w) & = & \E\left[|g_\sigma'(z)|^{2-d} \, |g_\sigma'(w)|^{2-d}
  \, G(Z_\sigma(z), Z_\sigma(w) ); \sigma < \infty \right] \\
 & = & \E\left[|g_\sigma'(z)|^{2-d} \, |g_\sigma'(w)|^{2-d}
 \, \lambda^{2(2-d)}
  \, G(\lambda
  Z_\sigma(z), \lambda Z_\sigma(w) ); \sigma < \infty
  \right] \\ & = & \E\left[|h'(z)|^{2-d} \, |h'(w)|^{2-d} \,
     G(\hat z, \hat w); \sigma < \infty \right] .
     \end{eqnarray*}
The Koebe $(1/4)$-theorem implies that
$             |h'(\zeta)| \asymp q^{-1} . $
   Distortion estimates
 (using Schwarz reflection if $y_w  \leq 2q$) imply that
 \[    |h'(z) | \asymp |h'(w)| \asymp |h'(\zeta)| \asymp q^{-1} ,
        \]\[  |\hat z - \hat w| \asymp 1 , \]
          \[     |\hat z|, |\hat w| \geq c , \]
          \[    \hat y_z \asymp (y_z \wedge q) \, q^{-1} , \;\;\;\; \hat y_w \asymp
          (y_w \wedge q)
           \, q^{-1} . \]
These estimates hold regardless of the value of  $S(\hat \zeta)$.
  If we also know that $S(\hat \zeta) \geq 1/10$, then
  \[     |\hat \zeta | \asymp |\hat z| \asymp | \hat w| \asymp 1.\]
   Hence, by Proposition \ref{sept9.prop1}, we see that
  \[   G(\hat z,\hat w) \leq c  \, \left[\frac{(y_z \wedge q) \, (y_w \wedge q)}
     {q^2} \right]^{4a-1}, \]
     \[     G(\hat z,\hat w) \geq c'  \, \left[\frac{(y_z \wedge q) \, (y_w \wedge q)}
     {q^2} \right]^{4a-1},\;\;\;\; \mbox{ if } S(\hat \zeta) \geq 1/10. \]
     Lemma \ref{oct23} implies that
     \[  \Prob\{\sigma < \infty\}   \asymp \Prob\{\sigma < \infty, S(\hat \zeta)
      \geq 1/10\}
      \asymp \left\{ \begin{array} {ll} y_w^{4a-1} \, (q/y_w)^{2-d} ,&   y_w \geq q \\
          q^{4a-1} , &y_w \leq q . \end{array} \right.  \]
Therefore,
\[  G(z,w) \asymp      y_w^{4a-1} \, (q/y_w)^{2-d} \,  q^{2(d-2)}\,
        \left[\frac{(y_z \wedge q) \, q}
     {q^2} \right]^{4a-1}
       , \;\;\;\; y_w \geq q, \]
     \[ G(z,w) \asymp  q^{4a-1} \, q^{2(d-2)}\,     \left[\frac{(y_z \wedge q) \, y_w}
     {q^2} \right]^{4a-1} , \;\;\;\; y_w \leq q. \]
     If $q \leq y_w \leq 2q$ we can use either  expression.
 If $y_w \leq 2q$, then $y_w \wedge q \asymp y_w,
     y_z \wedge q \asymp y_z, S(w) \vee q \asymp q$ and we can write
      \[  G(z,w)\asymp  q^{2(d-2)} \, q^{1-4a} \,   y_z^{4a-1} y_w^{4a-1}
\asymp q^{d-2} \, [S(w) \vee q]
  ^{-\beta} \, G(z) \, G(w)  . \]
  If $y_w \geq 2q$, then $y_z  \asymp y_w, y_z \wedge q \asymp q,
  S(w) \vee q \asymp y_w$, and we can write
  \[            G(z,w) \asymp y_w^{4a-1}\,q^{d-2} \,y_w^{d-2}
  = q^{d-2} \, y_w^{-\beta} \, y_w^{2(4a-1)}
   \asymp  q^{d-2} \, [S(w) \vee q]
  ^{-\beta} \, G(z) \, G(w). \]

\section{Proof of Theorem \ref{jul23.theorem1}}  \label{newsection}

By scaling and translation invariance,
we may assume that $z = 0$
and $\dist(z,\p D) = 1$. We first consider the case
$D = \Disk$, $w_1 = 1,$
and $w_2 = w = e^{2\theta i}$.  Note that
\[  G_\Disk(0;1,e^{2\theta i})
 =    S_{\Disk}(0;1,e^{2\theta i})^{4a-1} = \sin^{4a-1} \theta . \]

Let
$\Disk_t = e^{-t} \, \Disk$.
Let $\gamma$ be a chordal $SLE_\kappa$
path from $0$ to $w=e^{2i\theta}$ in
$\Disk$ and 
\[q(t,\theta) = [\sin\, \theta] ^{1-4a} \,
\Prob\left\{\dist(0,\gamma) \leq e^{-t} 
\right\}.\]
 We will use the
radial parametrization  normalized
so that at time
$t$, $|g_t'(0)| = e^{t}$. Here $g_t$
denotes the conformal transformation of
(the connected component containing the
origin) of $\Disk \setminus \gamma_t$ with
$g_t(0) = 0, g_t(\gamma(t)) = 1$.
We define $\theta_t$ by 
  $g_t(w) = e^{2\theta_t i} $.
  Under this parametrization, the local martingale
  is
  \[   M_t = e^{(2-d)t} \, [\sin \theta_t]^{4a-1} . \]
This parameterization ends at the time $T$ which is
the first time that $w$ is disconnected from $0$
by the curve $\gamma_T$; if $d \leq 4$, then $T$ is
the time at which $\gamma(T) = w$. Note that
$\dist(0,\gamma) = \dist (0,\gamma_T)$. 
  
 We will also consider two-sided radial $SLE$ which is
 the measure obtained by tilting by the local martingale
 $M_t$.  We recall some facts (see \cite[Lemmas 2.8 and 2.9]{LW}).
 The invariant probability density is 
  $f(\theta) = c \, \sin^{4a} \theta$.  Moreover, there
 exists $\alpha > 0$ such that if $f_t(\theta') = f_{t,\theta}(\theta')$
 denotes the density at time $t$, then for $t \geq 1$, 
\begin{equation}  \label{jul23.4}
f_t(\theta') = f(\theta') \, [1 + O(e^{-\alpha t})],
\end{equation}
 where the error term is bounded uniformly over the
 starting angle $\theta.$  Here $\alpha > 0$ is a constant
 that could be determined, but we will not need its exact value.
 Let
 \[       q(t) = \int_{0}^\pi q(t,\theta) \, f(\theta) \, d\theta.\]
 

\begin{proposition}  There exists $c < \infty$
such that for all $\theta$ and
all  $t \geq 1,s \geq 2$,
\[         q(s+ c e^{-s}) \,[1 - ce^{-\alpha t}]
 \leq    e^{t(2-d)}
  \, q(t+s,\theta)
    \leq q(s - c e^{-s})\,[1 + ce^{-\alpha t}].\]
 In particular,
 \[         q(s+ c e^{-s}) \,[1 - ce^{-\alpha t}]
 \leq    e^{t(2-d)}
  \, q(t+s)
    \leq q(s - c e^{-s})\,[1 + ce^{-\alpha t}].\]
\end{proposition}

\begin{proof}

The Koebe
$(1/4)$-theorem implies that the domain of
$g_t$ includes $ \Disk_{t + \log 4}$.
Using the distortion theorem, we see that if
$z \in \Disk_{t+2}$, then
\[ |g_t(z)|= e^{t} \,|z| [1 +O(e^{t}|z|)].\]
In particular, if $|z| = e^{-(s+t)},$
then
\[    |g_t(z)| = e^{-s} \, [1 + O(e^{-s})]
= \exp \left\{-s + O(e^{-s})\right\}.\]
Therefore, there exists $c_1$ such that
on the event $\{T > t\}$,
\begin{equation}  \label{jul21.3} 
 \Disk_{s + c_1 e^{-s}}
  \subset
     g_t(\Disk_{s+t}) \subset \Disk_{s - c_1e^{-s}} 
.
\end{equation}

Let $\xi =\xi_{s+t}  = \inf\{r: |\gamma(r)| = e^{-(t+s)}\}.$
Let
\[   Y(t,s) = \Prob\{\xi < T \mid
\gamma_{t \wedge T}\}. \]
Since $s \geq 2$, the Koebe $(1/4)$-theorem implies
that  $Y(t,s) = 0$ if $T \leq  t$.
The domain Markov property implies that
\[     q(t+s,\theta) = [\sin \theta]^{1-4a} \, \E\left[Y(t,s) \right]
 =  [\sin \theta]^{1-4a} \,\E \left[Y(t,s)\,;\, T > t\right]
  ,\]
  where here and below $\E$ denotes expectation with respect to chordal $SLE_\kappa$ from
  $0$ to $w$.  
  If $S_t = \sin \theta_t$, we know that
\[   M_t = e^{(2-d)t} \, S_t^{4a-1} \]
is a local martingale with $\Prob\{M_T
= 0\} = 1$.  Therefore,
\begin{eqnarray*}
\E \left[Y(t,s)\,;\, T > t  \right]
 & = & e^{(d-2)t} \,
 \E \left[Y(t,s)\,M_t \,
 S_t^{1-4a};\,T > t \right]\\
 & = & e^{(d-2)t} \, S_0^{4a-1}
  \, \E^*\left[Y(t,s) \, S_t^{1-4a} \right].
 \end{eqnarray*}
 Here $\E^*$ denotes the  measure
 obtained by tilting by $M$
  which is the same as the two-sided radial measure.  Again
 the initial $\theta$ is implicit in the
 notation.
 Hence,
 \[ q(t+s,\theta) =  e^{(d-2)t} 
  \, \E^*\left[Y(t,s) \, S_t^{1-4a} \right].\]
 By the strong Markov property, $
  Y(t,s)$ is
 the probability that a chordal $SLE$ from $0$
 to $e^{2i\theta_t}$ enters $g_t(\Disk_{t+s})$.
 Using \eqref{jul21.3}, we see that
for $s \geq 2$,
\[         q(s +c_1 e^{-s}, \theta_t)
 \leq   S^{1-4a}_t\,    Y(t,s) 
 \leq q(s -c_1 e^{-s},\theta_t). \]
From \eqref{jul23.4} we see that 
\[
 \E^*\left[Y(t,s) \, S_t^{1-4a} \right]
   \leq q(s-c_3e^{-s}) \, [1 + O(e^{-\alpha
    t})], \]
 \[
 \E^*\left[Y(t,s) \, S_t^{1-4a} \right]
   \geq q(s+c_3e^{-s}) \, [1 - O(e^{-\alpha
    t})]. \]
This completes the proof. \end{proof}

The next proposition finishes the proof of
Theorem \ref{jul23.theorem1} in the case
$D = \Disk$ with $u = \alpha/2$.

\begin{proposition}  There exist $\hat c,
c$ such
that for all $\theta$ and all $t \geq 2$,
   \[    |e^{t(2-d)} \, q(t,\theta)
         - \hat c| \leq c \, e^{-t \alpha/2}.\] 
\end{proposition}

\begin{proof}
By choosing $t' = t  \pm c' e^{-s}$
for large $c'$ in the last proposition, 
we see that 
\[ q(t+s,\theta) = e^{t(d-2)}
 \, q(s) \, [1 + O(e^{-\alpha t})
   + O(e^{- s})].\] 
 If   $L(t)
  = \log  [e^{t(2-d)}
 \, q(s)] $, then this implies that 
   \[  |L(t+s) - L(t)| \leq c \, [e^{-s}
    + e^{-\alpha t}].\]
 If $t \geq 2$ and $s \geq 0$, we have
\[ 
 |L(t+s) - L(t)|
= |L(t+s) - L(t/2)
  + L(t/2) - L(t) |\\
  \leq   c\, e^{-t\alpha/2}
.\]
This implies that $\lim_{t \rightarrow
\infty}L(t) =
L_\infty \in (-\infty,\infty)$ exists and
 \[  L(t) = L_\infty +O(e^{-t\alpha/2}).\]
This gives the result with $\hat c = e^{L_\infty}$.
\end{proof}

To finish the proof of Theorem \ref{jul23.theorem1}
 for general $D$ with $z=0,$ $ \dist(z, \p D)
= 1$, let $F:  \Disk \rightarrow D$ be a conformal transformation
with $F(0) = 0, F(1) = w_1, F(e^{2 \theta i}) = w_2.$
The $\theta$ depends on $D,w_1,w_2$, but conformal
invariance implies that
\[      S_D(0;w_1,w_2) = \sin \theta , \]
and hence
\[     G_D(0;w_1,w_2) = F'(0)^{2-d} \, [\sin \theta]^{4a-1} 
   = F'(0)^{2-d} \, G_\Disk(0;1,e^{2 \theta i}). \]
   The distortion theorem implies that there exists $c < \infty$
   such that 
   \[          [\lambda \epsilon  - c \epsilon^2]
    \, \Disk \subset F^{-1} ( \epsilon\Disk) \leq [\lambda \epsilon
     + c \epsilon^2] \, \Disk, \;\;\;\;
      \lambda = \frac {1}{F'(0)} \in [1/4,1] . \]
 Therefore, by conformal invariance,
\begin{eqnarray*}
 \Prob\{\dist(\gamma,0) \leq \epsilon\}
  & = & \hat c \, [\sin \theta]^{4a-1} \,  (\lambda \epsilon)^{2-d}
    \, [1 + O(\epsilon^u)]  \\
    & = &  \hat c \,  G_D(0;w_1,w_2) \, \epsilon^{2-d}
     \,  [1 + O(\epsilon^u)]  .
    \end{eqnarray*}


\begin{thebibliography}{00}



\bibitem{Bf} V. Beffara (2008).  The dimension of SLE curves, Annals of Probab.
{\bf 36},  1421-1452.\\ 

\bibitem{LJ1} F. Johansson Viklund and G. Lawler (2011). Optimal Holder exponent for the SLE path, Duke Math. J. {\bf 159 }, 351-383.\\



\bibitem{KS} I. Karatzas and S.Shreve. {\em Brownian motion and stochastic calculus, } Volume 113 of {\em Graduate texts in mathematics.} Springer-Verlag, New york, second edition, 1991.\\ 

\bibitem{LLN} S. Lalley, G. Lawler, H. Narayanan (2009).   Geometric interpretation
of half-plane capacity, Electron. Comm. Probab {\bf 14}, 566-571.\\ 


\bibitem{Law1} G. Lawler (2005). {\em Conformally Invariant Processes
in the Plane}, Amer. Math. Soc.\\ 


\bibitem{Law2} G. Lawler (2009). Schramm-Loewner evolution, in {\em statistical mechanics}, S.Sheffield and T. Spencer, ed., IAS/Park City Mathematical Series, AMS (2009), 231-295.\\ 

\bibitem{Law3} G. Lawler, Continuity of radial and two-sided radial SLE, preprint. 


\bibitem{LS} G. Lawler and S. Sheffield (2011). A natural parametrization for the Scheramm-Loewner evolution.   Annals  of Probab.
{\bf 39}, 1896--1937.\\


\bibitem{LSW} G. Lawler. O. Schramm, and W. Werner
(2004).  Conformal invariance of planar loop-erased
random walks and uniform spanning trees, Annals of
Probab. {\bf 32}, 939--995.\\ 
 

\bibitem{LW}  G. Lawler and  B. Werness.  Multi-point Green's function for SLE and an estimate of Beffara, to appear
in Annals of Probab.\\ 

\bibitem{LZ}  G. Lawler and W. Zhou, SLE curves and natural parametrization,  to appear
in Annals of Probab.\\ 




\bibitem{Lind}  J. Lind (2008).  H\"older regularity of the SLE trace, Trans. Amer. Math. Soc. 360 , {\bf 7}, 3557--3578.\\ 

\bibitem{RS} S. Rohde and O. Schramm (2005). Basic properties of
SLE, Annals of Math. {\bf 161}, 879--920.\\ 
 

\bibitem{Sch} O. Schramm (2000).
Scaling limits of loop-erased random walks
and uniform spanning trees, Israel J. Math. {\bf 118}, 221--288.\\ 
 

\bibitem{SS} O. Schramm and S. Sheffield (2005).
Harmonic explorer and its
convergence to SLE(4), Annals of Probab. {\bf 33}, 2127--2148. \\
 


\bibitem{Smir1} S. Smirnov (2001). Critical percolation in the plane: Conformal invariance, Cardy's formula, scaling limits. C. R. Acad. Sci. Paris Ser. I Math. {\bf 333} 239--244.\\
 

\bibitem{Smir2} S. Smirnov (2009). Conformal invariance in random cluster models. I. Holomorphic fermions in the Ising model, Ann of Math. {\bf 172} 1435--1467. 

\end{thebibliography}
\end{document}